\documentclass[a4paper,12pt]{article}
\usepackage{array,amsthm,amsfonts,epsfig}
\usepackage{graphicx}
\usepackage{dina42e}
\usepackage{amsmath}
\usepackage{mathdots}
\usepackage{amssymb}
\usepackage{textcomp}
\usepackage[T1]{fontenc}
\usepackage{fancyhdr}
\usepackage[utf8]{inputenc} 
\usepackage{euscript} 
\usepackage{mathrsfs}  
\usepackage{calligra} 
\usepackage{xcolor}

\newtheorem{satz}{Theorem}[section]
\newtheorem{thm}[satz]{Theorem}
\newtheorem{lemma}[satz]{Lemma}
\newtheorem{kor}[satz]{Corollary}
\newtheorem{prop}[satz]{Proposition}

\theoremstyle{definition}
\newtheorem{Def}[satz]{Definition}
\theoremstyle{remark}
\newtheorem{bem}[satz]{Remark}

\newcommand{\dq}{\textrm{\textit{\dj}} }                                           
\newcommand{\<}[1]{\langle #1 \rangle}                                    
\newcommand{\op}{OP}                                             
\newcommand{\e}{\varepsilon}                                              
\newcommand{\supp}{\textrm{supp }}                                        
\newcommand{\Sallg}[5]{S^{#1}_{#2,#3}(\R^{#4}\times\R^{#5})}              
\newcommand{\Sn}[3]{\Sallg{#1}{#2}{#3}{n}{n}}                             
\newcommand{\pa}[1]{\partial_{\xi}^{#1}}                                  
\newcommand{\p}{\partial}
\newcommand{\pax}[2]{\partial_{\xi}^{#1}\partial_{x}^{#2}}                
\newcommand{\R}{\mathbb{R}}                                               
\newcommand{\Rn}{\mathbb{R}^n}                                            
\newcommand{\RnRn}{\R^n \times \R^n}                                      
\newcommand{\RnRnx}[2]{\R^n_{#1} \times \R^n_{#2}}                         
\newcommand{\RnRnRnRn}{\RnRn \times \RnRn}                                
\newcommand{\intr}{\int \limits_{\Rn} }                                   
\newcommand{\osint}{\textrm{Os\hspace*{0,1cm}-}\hspace{-0,15cm}\iint}     
\newcommand{\osiint}{\textrm{Os\hspace*{0,1cm}-}\hspace{-0,15cm}\iiiint}  
\newcommand{\s}{\mathcal{S}(\R^n)}                                        
\newcommand{\sindo}[1]{\mathcal{S}(\R^n_{#1}) }                           
\newcommand{\sd}{\mathcal{S'}(\R^n)}                                      
\newcommand{\N}{\mathbb{N}}                                               
\newcommand{\Z}{\mathbb{Z}}                                               
\newcommand{\Non}{\mathbb{N}_0^n}                                               
\newcommand{\C}{\mathbb{C}}                                               

\title{Fredholm Property of Non-Smooth Pseudodifferential Operators}
\author{Helmut Abels\footnote{Fakult\"at f\"ur Mathematik,  
Universit\"at Regensburg,
93040 Regensburg,
Germany, e-mail: {\sf helmut.abels@mathematik.uni-regensburg.de}}\ \ and Christine Pfeuffer\footnote{Fakult\"at f\"ur Mathematik,  
Universit\"at Regensburg,
93040 Regensburg,
Germany}}

\begin{document}

\maketitle

\begin{abstract}
  In this paper we prove sufficient conditions for the Fredholm property of a non-smooth pseudodifferential operator $P$ which symbol is in a Hölder space with respect to the spatial variable. As a main ingredient for the proof we use a suitable symbol-smoothing. 
%
\end{abstract}

 {\bf AMS-Classification:} 35S05, 47B30,  47G30

\section{Introduction}

Fredholm operators are often called nearly invertible operators. Because of this reason they play an important role in the field of partial differential equations in order get existence and uniqueness results. Hence great effort already was spent to get some conditions for the Fredholmness of smooth pseudodifferential operators with symbols in the Hörmander-class $\Sn{m}{\rho}{\delta}:= \bigcap\limits_{M \in \N} S^m_{\rho, \delta}(\RnRn; M)$ where $0 \leq \rho, \delta \leq 1$ and $m \in \R$. The symbol-class $S^m_{\rho, \delta}(\RnRn; M)$ consists of all $M-$times continuous differentiable functions $a: \RnRn \rightarrow \C$ with are smooth with respect to the spatial variable such that for all $k \in \N_0$
\begin{align*}
  |a|^{(m)}_k := \max_{|\alpha|\leq  \min\{k, M\},|\beta|\leq k} \sup_{x, \xi \in \R^n}|\pax{\alpha}{\beta} a(x,\xi)|\<{\xi}^{-(m-\rho|\alpha|+\delta|\beta|)} < \infty.
\end{align*} 
For every symbol $a \in S^m_{\rho, \delta}(\RnRn; M)$  we define the associated pseudodifferential operator via
\begin{align}\label{Def1}
		\op (a)u(x):= a(x,D_x) u(x) := \intr e^{ix \cdot \xi} a(x,\xi) \hat{u}(\xi) \dq \xi \qquad \forall u \in \s , x \in \Rn,
\end{align}
where $\s$ denotes
the Schwartz space, i.\,e.\,, the space of all rapidly decreasing smooth functions and $\hat{u}$ is the Fourier transformation of $u$. In \cite{KohnNirenberg} Kohn and Nirenberg showed, that the ellipticity of a classical smooth pseudodifferential operator is necessary for its Fredholm property. Apart from necessary conditions Kumano-Go developed in \cite[Theorem 5.16]{KumanoGo} some sufficient conditions for the Fredholmness of smooth pseudodifferential operators: He showed that  pseudodifferential operators with so called \textit{slowly varying} smooth symbols of order $m$ are under certain conditions Fredholm operators form $H^m_2(\Rn)$ to $L^2(\Rn)$. Here $H^s_p(\Rn)$ denotes a Bessel Potential Space for $p \in (1, \infty)$ and $s \in \R$, defined in Section \ref{section:Preliminaries}. We refer to \cite[Definition 5.11]{KumanoGo} for the definition of the class of all smooth slowly varying symbols. In \cite{Schrohe1992} Schrohe extended the result of Kumano-Go as follows: Smooth pseudodifferential 
operators with slowly varying symbols of the order zero are Fredholm operators on the weighted Sobolev spaces $H^{st}_{\gamma}(\Rn)$, see \cite{Schrohe1992} for the definition, if and only if its symbol is uniformly elliptic.    
\\

In applications also partial differential equations with non-smooth pseudodifferential operators appear. Hence we are also interested in some sufficient conditions for non-smooth pseudodifferential operators to become a Fredholm operator from $H^m_p(\Rn)$, $m \in \N_0$, to $L^p(\Rn)$. For non-smooth differential operators the Fredholm property can be characterized under certain conditions by the ellipticity of its symbol. This was announced by Cordes in \cite{Cordes1972}, completed by Illner in \cite{Illner} and partially recovered by Fan and Wong in \cite{FanWong}. This characterization of the Fredholm property was extended to the matrix-valued case in \cite{Hoermander1979} for $p=2$ and in \cite{Sun} for general $p \in (1, \infty)$. In the case $p=2$ an alternative prove by means of the tool of $C^{\ast}-$algebras, was given by Talyor in \cite{Taylor1971}. The goal of this paper is to give sufficient conditions for the Fredholm property of pseudodifferential operators $a(x,D_x)$ with a symbol  $a$  in the 
non-smooth symbol-class $C^{\tau} 
S^0_{\rho,\delta}(\Rn \times \Rn, M)$, $0 \leq  \rho, \delta \leq 1$, $M \in \N_0 \cup \{\infty\}$. For the definition of the H\"older space $C^{\tau}$  with Hölder regularity $\tau>0$, $\tau \notin \N$ we refer to Section \ref{section:Preliminaries}. A function $a:\RnRn \rightarrow \C$ is an element of the symbol-class $C^{\tau} S^m_{\rho,\delta}(\Rn \times \Rn;M)$, $m \in \R$, if the following properties hold for all $\alpha, \beta \in \N_0^n$ with $|\beta| \leq \tau$ and $|\alpha| \leq M$:
  \begin{itemize}

		\item[i)] $\p_x^{\beta} a(x, .) \in C^{M}(\Rn)$ for all $x \in \Rn$,
		\item[ii)] $\p_x^{\beta} \pa{\alpha} a \in C^{0}(\R^n_x \times \R^n_{\xi})$,
		\item[iii)] $| \pa{\alpha} a(x,\xi)  | \leq C_{\alpha}\<{\xi}^{m-\rho|\alpha|}$ for all $x,\xi \in \R^n$
		\item[iv)] $\| \pa{\alpha} a(.,\xi)  \|_{C^{\tau}(\R^n)} \leq C_{\alpha}\<{\xi}^{m-\rho|\alpha|+\delta \tau}$ for all $\xi \in \R^n$.
	\end{itemize}
Moreover,  $a:\RnRn \rightarrow \mathcal{L}(\C^N)$ is an element of the symbol-class $C^{\tau} S^m_{\rho,\delta}(\Rn \times \Rn;M;\mathcal{L}(\C^N))$, $m \in \R$, $N\in\N$, if and only if $a_{j,k}\in C^{\tau} S^m_{\rho,\delta}(\Rn \times \Rn;M)$ for all $j,k=1,\ldots,N$, where we identify $A\in \mathcal{L}(\C^N))$ with a matrix $(a_{j,k})_{j,k=1}^N\in \C^{N\times N}$ in the standard way.
For a given symbol $a$ we define the associated pseudodifferential operator as in the smooth case, cf. (\ref{Def1}). We remark that in the literature there are also some results concerning the Fredholm propterty of pseudodifferential operators on compact manifolds, see e.g. \cite{Hoermander1994}, \cite{MolahajlooWong}. Nistor even gave some criteria for the Fredholmness of pseudodifferential operators on non-compact manifolds in \cite{Nistor}.\\

In the present paper we proceed as follows: We give a short summary of all notations and function spaces needed further on in Section \ref{section:Preliminaries}. Moreover we introduce the space of amplitudes and the oscillatory integrals, there. Section \ref{section:PDO} is devoted to define all symbol-classes of pseudodifferential operators needed later on and present its properties. In particular we extend the concept of symbol-smoothing presented in \cite[Section 1.3]{Taylor2}. Together with the extention of the symbol reduction result of \cite{Paper1} for non-smooth double symbols, see Subsection \ref{Subsection:SymbolReduction}, the symbol-smoothing becomes the main ingredient in order to verify the main result of our paper: 

\begin{thm}\label{thm:Fredholmproperty}
  Let $\tilde{m}, N \in \N$, $0<\tau <1$, $0\leq \delta <\rho \leq 1$, $m \in \R$, $M \in \N_0 \cup\{\infty\}$  and  $p \in (1,\infty) $ with $p=2$ if $\rho \neq 1$. Additionally we choose an arbitrary $\theta \in \left( 0; \min\{ (\tilde{m}+\tau)(\rho-\delta); 1\}\right)$ and $\tilde{\e} \in \left( 0, \min\{ (\rho-\delta)\tau; (\rho-\delta)(\tilde{m}+\tau)-\theta; \theta)\} \right)$. Moreover let $a \in C^{\tilde{m},\tau}\tilde{S}^{m}_{\rho,\delta}(\RnRn;M;\mathcal{L}(\C^N))$ be a symbol fulfilling the following properties for some $R>0$ and $C_0>0$:
  \begin{itemize}
    \item[1)] $|\det(a(x,\xi))|\<{\xi}^{-m}\geq C_0$ for all $x,\xi \in \Rn$ with $|x|+|\xi|\geq R$.
    \item[2)] $a(x, \xi) \xrightarrow{|x| \rightarrow \infty} a(\infty, \xi)$ for all $\xi \in \Rn$.
   \end{itemize}
   Then for all  $M \geq (n+2)+n\cdot\max\{1/2, 1/p\}$ and $s\in \R$ with $$(1-\rho)\frac{n}{2}-(1-\delta)(\tilde{m}+\tau)+\theta + \tilde{\e}<s<\tilde{m}+\tau$$
  the operator
  \begin{align*}
    a(x,D_x): H^{m+s}_p(\Rn)^{ N} \rightarrow H^s_p(\Rn)^{ N}
  \end{align*}
  is a Fredholm operator.
\end{thm}

This theorem will be proved in Section \ref{FredholmProperty}. For the definition of the symbol-class $C^{\tilde{m},\tau}\tilde{S}^{m}_{\rho,\delta}(\RnRn;M ; \mathcal{L}(\C^N))$ we refer to Definition \ref{Def:SlowlyVaryingSymbols} in Subsection \ref{Subsection:SymbolSmoothing}.

\section{Notations and Function Spaces}\label{section:Preliminaries}

The set of all natural numbers without $0$ is denoted by $\N$. Unless otherwise noted we consider $n \in \N$ during the whole paper. We define
$$\<{x}:=(1+|x|^2)^{1/2} \quad \text{ for each } x \in \Rn \qquad \text { and } \qquad  \dq \xi:= (2 \pi)^{-n} d \xi.$$
Moreover
$$ \<{x;y}:= (1+|x|^2+ |y|^2)^{1/2} \qquad \text{for all } x, y \in \Rn.$$
Additionally we set for each $x \in \R$
$$\lfloor x \rfloor:= \max\{ l \in \Z: l \leq x \} \qquad \text{ and } \qquad \lceil x \rceil := \min\{l \in \Z: l \geq x\}.$$
For each multi-index $\alpha=(\alpha_1, \ldots, \alpha_n) \in \Non$ we use the notations $ \p^{\alpha}_x := \p^{\alpha_1}_{x_1} \ldots \p^{\alpha_n}_{x_n} $ and $D^{\alpha}_x:= (-i)^{|\alpha|}\p^{\alpha}_x$. 

Assuming two Banach spaces $X,Y$ the set of all linear and bounded operators $A:X \rightarrow Y$ is denoted by $\mathscr{L}(X,Y)$. In case $X=Y$, we just write $\mathscr{L}(X)$.

For  $s \in (0,1)$ the set of all functions $f: \Rn \rightarrow \C$ fulfilling
\begin{align*}
  \|f\|_{C^{0,s}} \equiv \| f\|_{C^{0,s}(\Rn)} := \sup_{x \in \Rn}|f(x)| +\sup_{x \neq y} \frac{|f(x)-f(y)|}{|x-y|^s}< \infty
\end{align*}
is called \textit{Hölder space} $C^{0,s}(\Rn)$ of the order $0$ with Hölder continuity exponent $s$. A function $f: \Rn \rightarrow \C$  is in the Hölder space $C^{\tilde{m},s}(\Rn)$ of the order $\tilde{m} \in \N_0$, also denoted by $C^{\tilde{m}+s}(\Rn)$, if  we have $\p^{\alpha}_x f \in C^{0,s}(\Rn)$ for each $\alpha \in \Non$ with $|\alpha| \leq \tilde{m}$. Note that all Hölder spaces are Banach spaces. 

On account of the definition of the H\"older spaces and the Leibniz-rule we obtain:
\begin{lemma}\label{lemma:PropertyHoelderSpaces}
  Let $\tilde{m} \in \N_0$, $0< \tau < 1$ and $f,g \in C^{\tilde{m},\tau}(\Rn)$. Then 
  \begin{align*}
    \|fg\|_{C^{\tilde{m},\tau}} \leq \sum_{\tilde{m}_1+\tilde{m}_2=\tilde{m}} C_{\tilde{m}}\left\{ \|f\|_{C^{\tilde{m}_1}_b} \|g\|_{C^{\tilde{m}_2, \tau}} + \|f\|_{C^{\tilde{m}_1, \tau}} \|g\|_{C^{\tilde{m}_2}_b} \right\}.
  \end{align*}
\end{lemma}

%
%


 The \textit{Bessel Potential space} $H^s_p(\Rn)$, $s \in \R$ and $1<p<\infty$, will play a central role in this paper. The set $H^s_p(\Rn)$ is defined by 
\begin{align}\label{BesselPotentialSpace} 
  H^s_p(\Rn):= \{ f \in \sd: \<{D_x}^s f \in L^p(\Rn) < \infty \},
\end{align}
where $\<{D_x}^s:=\op(\<{\xi}^s) $.\\




Finally let us mention for the convenience of the reader an interpolation result needed in this paper, see e.g. \cite[Lemma 2.41]{Diss}:

\begin{lemma}\label{lemma:InterpolationResult}
  Let $k, m \in \N$ with $k \leq m$, $0<\tau < 1$ and $\theta:= \frac{k}{m + \tau}$. Then 
  \begin{align*}
    \|f\|_{C^k_b(\Rn)} \leq C\|f\|^{1-\theta}_{C^0_b(\Rn)} \|f\|^{\theta}_{C^{m, \tau}(\Rn)} \qquad \text{for all } f \in C^{m, \tau}(\Rn).
  \end{align*}

\end{lemma}

\subsection{Space of Amplitudes and Oscillatory Integrals}\label{subsection:SpaceOfAmplitudes}

In the present paper we need \textbf{oscillatory integrals} defined by
\begin{align} \label{DefOsziInt}
  \osint e^{-iy \cdot \eta} a(y,\eta) dy \dq \eta := \lim_{\e \rightarrow 0} \iint \chi(\e y, \e \eta) e^{-iy \cdot \eta} a(y,\eta) dy \dq \eta
\end{align}
for all elements $a$ of the \textit{space of amplitudes} $\mathscr{A}^{m,N}_{\tau,M}(\RnRn)$, $ N,M \in \N_0 \cup \{ \infty \}$, $m,\tau \in \R$, where $\chi \in \mathcal{S}(\RnRn)$ with $\chi(0,0)=1$. Here $\mathscr{A}^{m,N}_{\tau,M}(\RnRn)$, $ N,M \in \N_0 \cup \{ \infty \}$, $m,\tau \in \R$ is the set of all continuous functions $a:\Rn \times \Rn \rightarrow \C$ such that for all
$\alpha, \beta \in \Non$ with $|\alpha| \leq N$, $|\beta| \leq M$ we have
  \begin{enumerate}
    \item[i)] $\p^{\alpha}_{\eta} \p^{\beta}_{y} a(y,\eta) \in C^0(\RnRnx{y}{\eta})$,
    \item[ii)] $\left|\p^{\alpha}_{\eta} \p^{\beta}_{y} a(y, \eta) \right| \leq C_{\alpha, \beta} (1 + |\eta|)^m (1 + |y|)^{\tau}$ for all $y, \eta \in \Rn$.
  \end{enumerate} 


Defining for all $m \in \N$
\begin{align}\label{eqDefAg}
  A^m(D_{x},\xi) &:= \<{\xi}^{-m} \<{D_x}^{m} \ &\text{ if } m \text{ is even},\\ \label{eqDefAu}
  A^m(D_{x},\xi) &:= \<{\xi}^{-m-1} \<{D_x}^{m-1} -\sum_{j=1}^n \<{\xi}^{-m} \frac{\xi_j}{\<{\xi} } \<{D_x}^{m-1} D_{x_j} \ &\text{ else},
\end{align}
we can extend some properties of the oscillatory integral proved in Section 2.3 of \cite{Paper1} as follows:

\begin{thm}\label{thm:propertiesOsciInt}
  Let $m, \tau \in \R$ and $N,M \in \N_0 \cup \{ \infty \}$ with $N>n + \tau$. Moreover let $l, l' \in \N$ with $N \geq l'> n+\tau$ and $M \geq l > n+m$. 
  Then the oscillatory integral (\ref{DefOsziInt}) exists for all $a \in \mathscr{A}^{m,N}_{\tau,M}(\RnRn)$ and we have for all $l_1,l_2 \in \N_0$ with $l_1\leq N$ and $l_2 \leq l$:
  \begin{align*}
    \osint e^{-iy\cdot \eta} a(y,\eta) dy \dq \eta &= \iint e^{-iy\cdot \eta} A^{l'}(D_{\eta}, y) A^l(D_y, \eta) a(y,\eta) dy \dq \eta \\
     \osint e^{-iy\cdot \eta} a(y,\eta) dy \dq \eta &=  \osint e^{-iy\cdot \eta}A^{l_1}(D_{\eta}, y) A^{l_2} (D_y, \eta) a(y,\eta) dy \dq \eta
  \end{align*}
\end{thm}
\begin{proof}
  The claim can be verified in the same way as in Theorem 2.10 and Theorem 2.12 of \cite{Paper1}, if one takes care of ii) just holding for $|\beta| \leq l$.
\end{proof}

\begin{thm}
  Let $m, \tau \in \R$, $m_i, \tau_i \in \R$ for $i\in \{1,2\}$ and $N \in \N_0 \cup \{ \infty \}$ such that there is a $l' \in \N$ with  $N\geq l'>n + \tau$. Moreover let $\alpha, \beta \in\Non$ with $|\alpha| \leq \tilde{M}$, where $\tilde{M}:= \max \{ \hat{m} \in \N_0: N-\hat{m} >n+\tau \}$ and $l \in \N$ with $l>m+n$. Considering $a \in C^0(\Rn_y \times \Rn_{y'} \times \Rn_{\eta} \times \Rn_{\xi})$ with 
  \begin{itemize}
    \item $\left| A^{l'}(D_{\eta}, y) A^l(D_y, \eta) a(y, y', \eta, \xi) \right| \leq C_{l, l'} \<{y}^{\tau-l'} \<{\eta}^{m-l} \<{y'}^{\tau_1} \<{\xi}^{m_1}$
    \item $\left| A^{l'}(D_{\eta}, y) A^l(D_y, \eta) \pa{\alpha} \p_{y'}^{\beta} a(y, y', \eta, \xi) \right| \leq C_{l, l', \alpha, \beta} \<{y}^{\tau-l'} \<{\eta}^{m-l} \<{y'}^{\tau_2} \<{\xi}^{m_2}$ 
  \end{itemize}
  for all $y,y', \eta, \xi \in \Rn$ we have for all  $y', \xi \in \Rn$:
  \begin{align*}
     \pa{\alpha} \p_{y'}^{\beta} \osint e^{-iy \cdot \eta} a(y, y', \eta, \xi) dy \dq \eta 
      = \osint e^{-iy \cdot \eta}  \pa{\alpha} \p_{y'}^{\beta} a(y, y', \eta, \xi) dy \dq \eta.
  \end{align*}

\end{thm}
\begin{proof}
  This result can be verified similarly to \cite[Theorem 2.11]{Paper1}.
\end{proof}

\begin{kor}
  Let $m, \tau \in \R$ and $N \in \N_0 \cup \{ \infty \}$ such that there is some $l' \in \N$ with  $N\geq l'>n + \tau$. Moreover let $l \in \N$ with $l>n+m$. Additionally let $a_j, a \in C^0(\RnRn)$, $j \in \N_0$ such that for all $\alpha, \beta \in\Non$ with $|\alpha| \leq N$ and $|\beta| \leq l$ the derivatives $\p_{\eta}^{\alpha} \p_y^{\beta} a_j, \p_{\eta}^{\alpha} \p_y^{\beta} a$ exist in the classical sense and 
  \begin{itemize}
    \item $|\p_{\eta}^{\alpha} \p_y^{\beta} a_j(y, \eta)| \leq C_{\alpha, \beta} \<{\eta}^{m} \<{y}^{\tau}$ for all $\eta, y \in \Rn$, $j \in \N_0$,
    \item $|\p_{\eta}^{\alpha} \p_y^{\beta} a| \leq C_{\alpha, \beta} \<{\eta}^{m} \<{y}^{\tau}$ for all $\eta, y \in \Rn$,
    \item $\p_{\eta}^{\alpha} \p_y^{\beta} a_j(y, \eta) \xrightarrow{j \rightarrow \infty } \p_{\eta}^{\alpha} \p_y^{\beta} a(y, \eta)$ for all $\eta, y \in \Rn$.
  \end{itemize}
  Then 
  \begin{align*}
    \lim_{j \rightarrow \infty} \osint e^{-iy \cdot \eta}  a_j(y, \eta) dy \dq \eta 
    =  \osint e^{-iy \cdot \eta}  a(y, \eta) dy \dq \eta .
  \end{align*}
\end{kor}
\begin{proof}
  The claim can be shown similarly to \cite[Corollary 2.13]{Paper1}.
\end{proof}

Another property of oscillatory integral needed later on is

\begin{bem}\label{bem:PropOsciIntBasic}
  Assuming $u \in C^{\infty}_b(\Rn)$ and $x \in \Rn$ we obtain
  \begin{align*}
    \osint e^{i(x-y)\cdot \eta} u(y)dy \dq \eta 
    = u(x).
  \end{align*}
\end{bem}
For the proof see e.g. \cite[Example 3.11]{PDO}.

\section{Pseudodifferential Operators and their Properties}\label{section:PDO}

Throughout this section we summarize all properties of pseudodifferential operators needed later on. Additionally we define all symbol-classes of pseudodifferential operators needed in this paper.

As shown in \cite[Remark 4.2]{Diss} we have
 \begin{align}\label{glattesSymbolIstNichtglatt}
      S^{m}_{\rho, \delta}(\Rn \times \R^n; M) \subseteq C^{\tilde{m},s} S^{m}_{\rho, \delta}(\Rn \times \R^n;M).
  \end{align}
for all $0<s \leq 1$, $\tilde{m} \in \N_0$,  $m\in \R$, $M \in \N_0 \cup \{ \infty \}$ and $0 \leq \rho,\delta \leq 1$.

Additionally we get by means of interpolation, c.f. Lemma \ref{lemma:InterpolationResult}, the next estimate for non-smooth symbols:

\begin{bem}\label{bem:AbschatzungNichtglattesSymbol}
  Let $\tilde{m} \in \N_0$, $0<\tau < 1$, $0 \leq \delta, \rho \leq 1$, $m \in \R$ and $a \in C^{\tilde{m}, \tau} S^m_{\rho, \delta}(\RnRn; M)$. 
  Then we get for all $\alpha \in \Non$ with $|\alpha| \leq M$ and $k \in \N_0$ with $k \leq \tilde{m}$:
  \begin{align*}
    \|\pa{\alpha} a(., \xi)\|_{C^k_b(\Rn)} \leq C_{\alpha, \beta} \<{\xi}^{m-\rho|\alpha| + \delta k} \qquad \text{for all } \xi \in \Rn.
  \end{align*}
\end{bem}

Pseudodifferential operators are bounded as maps between several Bessel Potential spaces. For the proof we refer to \cite[Theorem 3.7]{Paper1}.  

\begin{thm}\label{thm:BoundednessResultNonSmooth}
  Let $m \in \R$, $0 \leq \delta \leq \rho \leq 1$ with $\rho >0$, $1<p < \infty$ and $M \in \N_0 \cup \{\infty\}$ with $M>max \{\frac{n}{2}, \frac{n}{p} \}$. Additionally let $\tau > \frac{1-\rho}{1-\delta} \cdot \frac{n}{2}$ if $\rho <1$ and $\tau >0$ if $\rho =1$ respectively. Moreover let $\mathscr{B} \subseteq C^{\tau} S^{m-k_p}_{\rho, \delta} (\RnRn; M)$ be bounded. Denoting $k_p:=(1-\rho) n \left| 1/2 - 1/p \right|$ and let $(1-\rho)n/p - (1-\delta) \tau < s< \tau$ there is some $C_{s}>0$, independent of $a \in \mathscr{B}$,  such that 
  \begin{align*}
    \left\| a(x,D_x) f \right\|_{H^s_p(\Rn)} \leq C_{s} \| f\|_{H_p^{s+m}(\Rn)} \qquad \text{for all } a \in \mathscr{B} \text{ and } f \in H_p^{s+m}(\Rn).
  \end{align*}
\end{thm}

\subsection{Symbol-Smoothing}\label{Subsection:SymbolSmoothing}

A well-known tool for proving some properties of non-smooth pseudodifferential operators of the symbol class $XS^m_{1,\delta}(\RnRn)$ for certain Banach spaces $X$ is the symbol-smoothing, see e.g. \cite[Section 1.3]{Taylor2}. 
In order to prove the Fredholm property of non-smooth pseudodifferential operators, we now generalize the tool of  symbol-smoothing for pseudodifferential operators which are non-smooth with respect to the second variable and for $\rho \neq 1$. To this end we fix two functions $\phi, \psi_0 \in C^{\infty}_0(\Rn)$ till the end of this section with the following properties:
\begin{itemize}
  \item $\phi(\xi)=1$ for all $|\xi|\leq 1$,
  \item $\psi_0 \geq 0$, $\psi_0(\xi)=1$ for all $|\xi| \leq 1$ and $\psi_0(\xi)=0$ for all $|\xi| \geq 2$.
\end{itemize}
Then we define for all $j \in \N$ the functions $\psi_j$ via
\begin{align*}
  \psi_j(\xi):= \psi_0(2^{-j}\xi)- \psi_0(2^{-j-1} \xi) \qquad \text{for all } \xi \in \Rn.
\end{align*}
Using that for any $a \in \R$ there are $C_1, C_2>0$ such that
\begin{align}\label{8e}
  C_1 \<{\xi}^{-a} \leq 2^{-ja} \leq C_2 \<{\xi}^{-a} \qquad \text{for all } \xi \in \supp(\psi_j), j \in \N
\end{align}
 we can show  the following properties of the  functions $\psi_j$ for all $\alpha \in \Non$:
\begin{align}\label{7e}
  \|\p^{\alpha}_{\xi}\psi_j\|_{\infty} \leq C_{\alpha}\<{\xi}^{-|\alpha|}.
\end{align}
 Additionally we define for all $\e>0$ the operator $J_{\e}$ by
\begin{align*}
  J_{\e}:= \phi(\e D_x).
\end{align*}
Note, that  for each $\alpha \in \Non$:
\begin{align}\label{30e}
  \pa{\alpha} J_{\e} = J_{\e} \pa{\alpha}.
\end{align}

The operator $J_{\e}$ has the following properties:

\begin{lemma}\label{lemma:PropertiesOfJ_E}
  For $\e, s >0$ with $s \notin \N$ we have for all $f \in C^s(\Rn)$:
  \begin{itemize}
    \item[i)] $\|D_x^{\beta} J_{\e} f\|_{\infty} \leq C \|f\|_{C^s(\Rn)}$ for all $\beta \in \Non$ with $|\beta| \leq s$,
    \item[ii)] $\|D_x^{\beta} J_{\e} f\|_{\infty} \leq C \e^{-(|\beta|-s)}\|f\|_{C^s(\Rn)}$ for all $\beta \in \Non$ with $|\beta| > s$, 
    \item[iii)] $\|D_x^{\beta} \left(1-J_{\e} \right)f \|_{C^{s-|\beta|-t}(\Rn)} \leq C \e^{t} \|f\|_{C^{s}(\Rn)}$ for all $\beta \in \Non$ with $|\beta| \leq s$ and $t \geq 0$ with $s-t-|\beta| >0$ and $s-t-|\beta| \notin \N$,
    \item[iv)] $\|D_x^{\beta} \left(1-J_{\e} \right)f \|_{\infty} \leq C_{s} \e^{s-|\beta|} \|D_x^{\beta}f\|_{C^{s-|\beta|}(\Rn)}$  for all $\beta \in \Non$ with $|\beta| \leq s$.
  \end{itemize}
\end{lemma}
\begin{proof}
  On account of \cite[Lemma 1.3C]{Taylor2} the claims i), ii) and claim iv) in the case $|\beta| =0$ hold true. An application of the case $|\beta|=0$ on $g := D_x^{\beta}f \in C^{s-|\beta|}(\Rn)$ provides the general case of claim iv). Because of \cite[Lemma 1.3.A]{Taylor2} we additionally obtain  claim iii) for the case $|\beta| =0$. It remains to verify claim iv) for general  $\beta \in \Non$ with $|\beta| \leq s$. This can be done similarly to the proof of the case $|\beta|=0$. For the convenience of the reader we give a short proof of claim iii) for arbitrary  $\beta \in \Non$ with $|\beta| \leq s$, now. 
  Due to the boundedness of $\{ \e^{-t}\<{\xi}^{-t} (1-\phi(\e \xi)): \e \in (0,1] \} \subseteq S^0_{1,0}(\Rn_x \times \Rn_{\xi}) $ and due to $\frac{\xi^{\beta}}{\<{\xi}^{|\beta|}} \in S^0_{1,0}(\Rn_x \times \Rn_{\xi})$ we get the boundedness of 
  \begin{align*}
    \left\{ \e^{-t} \xi^{\beta}\<{\xi}^{-t} (1-\phi(\e \xi)): \e \in (0,1] \right\} \subseteq S^{|\beta|}_{1,0}(\Rn_x \times \Rn_{\xi}). 
  \end{align*}
  Since $\<{D_x}^{-t}$ and $D_x^{\beta}$ commute, we obtain claim iii) in the general case. 
\end{proof}

\begin{Def}\label{Def:SymbolSmoothing}
  Let $\tilde{m} \in \N_0$, $0<\tau <1$, $M \in \N_0 \cup \{ \infty\}$ $m \in \R$ and $0 \leq \delta \leq \rho \leq 1$. For $\gamma \in (\delta, 1)$ we set $\e_j:= 2^{-j\gamma}$. For each $a \in C^{\tilde{m}, \tau} S^{m}_{\rho, \delta}(\RnRn; M)$ we define 
  \begin{itemize}
    \item $a^{\sharp}(x, \xi):= \sum\limits_{j=0}^{\infty} J_{\e_j} a(x, \xi)\psi_j(\xi)$ for all $x, \xi \in \Rn$,
    \item $a^b(x, \xi):= a(x, \xi)-a^{\sharp}(x, \xi)$ for all $x, \xi \in \Rn$.
  \end{itemize}
\end{Def}

Our aim is to verify useful properties of the  functions $a^{\sharp}$ and $a^b$ needed later on. To this end two new symbol-classes are needed, which we define, now. 

\begin{Def}\label{Def:SlowlyVaryingSymbols}
  Let $\tilde{m} \in \N_0$, $0 < \tau < 1$, $m \in \R$, $0 \leq \delta, \rho \leq 1$ and $M \in \N_0 \cup \{ \infty \}$. 
  Then  $a \in C^{\tilde{m}, \tau} S^m_{\rho, \delta} (\RnRn; M)$ belongs to the symbol-class $C^{\tilde{m}, \tau} \dot{S}^m_{\rho, \delta} (\RnRn; M)$, if for all $\alpha, \beta \in \Non$ with $|\alpha| \leq M$ and $|\beta| \leq \tilde{m}$ we have
  \begin{align*}
    |\pa{\alpha} D_x^{\beta} a(x, \xi)| \leq C_{\alpha, \beta}(x) \<{\xi}^{m-\rho|\alpha|+\delta|\beta|} \qquad \text{for all } x, \xi \in \Rn,
  \end{align*}
  where $C_{\alpha, \beta}(x)$ is a bounded function, which converges to zero, as $|x| \rightarrow \infty$. \\
  Moreover, $a \in C^{\tilde{m}, \tau} S^m_{\rho, \delta} (\RnRn; M)$ belongs to the symbol-class $C^{\tilde{m}, \tau} \tilde{S}^m_{\rho, \delta} (\RnRn; M)$, if for all $\beta \in \Non$ with $|\beta| \leq \tilde{m}$ and $|\beta | \neq 0$ we have
  \begin{align*}
    D_x^{\beta} a(x, \xi) \in C^{\tilde{m}-|\beta|, \tau} \dot{S}^{m+\delta|\beta|}_{\rho, \delta} (\RnRn; M).
  \end{align*}
  We call the elements of  $C^{\tilde{m}, \tau} \tilde{S}^m_{\rho, \delta} (\RnRn; M)$ \textbf{slowly varying symbols}.
  Moreover,  $a:\RnRn \rightarrow \mathcal{L}(\C^N)$ is an element of the symbol-class $C^{\tau}  \dot{S}^m_{\rho,\delta}(\Rn \times \Rn;M;\mathcal{L}(\C^N))$ respectively $C^{\tau}  \tilde{S}^m_{\rho,\delta}(\Rn \times \Rn;M;\mathcal{L}(\C^N))$, $m \in \R$, $N\in\N$, if and only if $a_{j,k}\in C^{\tau} \dot{S}^m_{\rho,\delta}(\Rn \times \Rn;M)$ respectively $a_{j,k}\in C^{\tau} \tilde{S}^m_{\rho,\delta}(\Rn \times \Rn;M)$ for all $j,k=1,\ldots,N$, where we identify $A\in \mathcal{L}(\C^N))$ with a matrix $(a_{j,k})_{j,k=1}^N\in \C^{N\times N}$ in the standard way.
\end{Def}

The properties of the functions $a^{\sharp}$ and $a^b$ are summarized in the next three lemmas:

\begin{lemma}\label{lemma:SymbolSmoothing3}
  Let $0 \leq \delta < \rho \leq 1$, $\tilde{m} \in \N$, $0<\tau <1$, $M \in \N \cup \{ \infty \}$, $m \in \R$ and $a \in C^{\tilde{m}, \tau} S^m_{\rho, \delta}(\RnRn; M)$. Moreover let $\gamma \in (\delta, \rho)$. Then we have for $\tilde{\e} \in (0, (\gamma-\delta)\tau))$:
  \begin{itemize}
      \item[i)] $D_x^{\beta} a^{b}(x, \xi) \in C^{\tilde{m}-|\beta|, \tau} S^{m-(\gamma-\delta)(\tilde{m}+\tau)+\gamma|\beta|}_{\rho, \gamma}(\RnRn; M) \forall \beta \in \Non$ with $|\beta| \leq \tilde{m}$,
      \item[ii)] $a^{b}(x, \xi) \in C^{\tilde{m}, \tau} \tilde{S}^{m-(\gamma-\delta)(\tilde{m}+\tau) + \tilde{\e}}_{\rho, \gamma}(\RnRn; M)$ if $a \in C^{\tilde{m}, \tau} \tilde{S}^m_{\rho, \delta}(\RnRn; M)$,
      \item[iii)] $a^{b}(x, \xi) \in C^{\tilde{m}, \tau} \dot{S}^{m-(\gamma-\delta)(\tilde{m}+\tau) + \tilde{\e}}_{\rho, \gamma}(\RnRn; M)$  if $a \in C^{\tilde{m}, \tau} \dot{S}^m_{\rho, \delta}(\RnRn; M)$.
  \end{itemize}
\end{lemma}
\begin{proof}
  We begin with the proof of $i)$. We choose an arbitrary $\xi \in \Rn$ and set $N:= \{j \in \N_0: \xi \in \supp \psi_j \}$. Then $\sharp N \leq 5$. Using $a^{\sharp}(.,\xi) = \sum_{j \in N} J_{\e_j}a(.,\xi) \psi_{j}(\xi)$ and the Leibniz rule yields for all $\alpha, \beta \in \Non$ with $|\alpha| \leq M$ and $|\beta| \leq \tilde{m}$
  \begin{align*}
    |\pa{\alpha} D_x^{\beta}a^b(x,\xi)|
    &= \left| \pa{\alpha} D_x^{\beta} \sum_{j=0}^{\infty}(1-J_{\e_j})(p(x, \xi)\psi_j(\xi)) \right| \\
    &\leq \sum_{j \in N} \sum_{\alpha_1  + \alpha_2= \alpha} C_{\alpha} \| (1-J_{\e_j})(\pa{\alpha_1}D_x^{\beta} p(x, \xi) \pa{\alpha_2} \psi_j(\xi) )\|_{L^{\infty}(\Rn_x)}
  \end{align*}
  An application of Lemma \ref{lemma:PropertiesOfJ_E} iv), (\ref{8e}) and (\ref{7e}) to the previous estimate provides:
  \begin{align}\label{32e}
    |\pa{\alpha} D_x^{\beta}a^b(x,\xi)|
    &\leq \sum_{j \in N} \sum_{\alpha_1  + \alpha_2=\alpha} C_{\alpha} \e_j^{\tilde{m}+\tau-|\beta|} \|(\pa{\alpha_1} D_x^{\beta} p(x, \xi) \pa{\alpha_2} \psi_j(\xi) )\|_{C^{\tilde{m}+\tau-|\beta|}(\Rn_x)} \notag\\
    &\leq \sum_{j \in N} \sum_{\alpha_1  + \alpha_2=\alpha} C_{\alpha} \<{\xi}^{-\gamma(\tilde{m}+\tau-|\beta|)} |\pa{\alpha_2} \psi_j(\xi)| \|\pa{\alpha_1} D_x^{\beta} p(x, \xi)  \|_{C^{\tilde{m}+\tau-|\beta|}(\Rn_x)} \notag\\
    &\leq C_{\alpha, \tilde{m}, \tau} \<{\xi}^{m-(\gamma-\delta)(\tilde{m} + \tau) + \gamma|\beta|-\rho|\alpha|} \qquad \text{for all } x,\xi \in \Rn.
  \end{align}

  Similarly we get by means of (\ref{30e}), the Leibniz rule, Lemma \ref{lemma:PropertiesOfJ_E} iii) and (\ref{7e})  for all $\alpha, \beta \in \Non$ with $|\alpha|\leq M$ and $|\beta| \leq \tilde{m}$:
  \begin{align}\label{31e}
    \|\pa{\alpha}D_x^{\beta} a^b(., \xi)\|_{C^{\tilde{m}-|\beta|, \tau}(\Rn)} \leq C_{\alpha, \beta} \<{\xi}^{m-(\gamma-\delta)(\tilde{m}+\tau)  +\gamma |\beta| -\rho|\alpha| + \gamma(\tilde{m}-|\beta| + \tau)} 
  \end{align}
  for all $\xi \in \Rn$. 
  On account of (\ref{31e}) and (\ref{32e}) claim i) holds. 
  
  Our next goal is show $ii)$ and $iii)$. 
  In order to prove the claim, we assume $a \in C^{\tilde{m}, \tau} \dot{S}^m_{\rho, \delta}(\RnRn; M)$ or $a \in C^{\tilde{m}, \tau} \tilde{S}^m_{\rho, \delta}(\RnRn; M)$. Additionally we fix some arbitrary $\alpha, \beta \in \Non$ with $|\alpha| \leq M$, $|\beta| \leq \tilde{m}$ and $|\beta| \neq 0$ if $a \in C^{\tilde{m}, \tau} \tilde{S}^m_{\rho, \delta}(\RnRn; M)$. We choose an arbitray $\e >0$. As before we fix an arbitrary $\xi \in \Rn$ and set $N:= \{j \in \N_0: \xi \in \supp \psi_j \}$. Moreover we define for all $j \in \N_0$ the functions $\varphi_{\e_j}, g_{\e_j}, g:\Rn \rightarrow \C$ via
  \begin{itemize}
    \item $\varphi_{\e_j}:= \delta_0- \mathscr{F}^{-1}_{\xi \rightarrow x}\left[ \phi(\e_j \xi) \right] $ in $\sd$,
    \item $g_{\e_j}(x):= \mathscr{F}^{-1}_{\xi \rightarrow x}\left[ \phi(\e_j \xi) \right](x)$ for all $x \in \Rn,$
    \item $g(x):=\mathscr{F}^{-1}_{\xi \rightarrow x}\left[ \phi(\xi) \right](x)$ for all $x \in \Rn$.
  \end{itemize}
  By means of integration by parts and the Theorem of Fubini, we obtain for each $j \in \N$
  \begin{align}\label{39e}
    \left[ 1-\phi (\e_j D_x) \right]f = \varphi_{\e_j} \ast f(x) \qquad \text{for all } f \in C^0_b(\Rn).
  \end{align}
  Since we can change the order of the two operators $D_x^{\beta}$ and $(1-J_{\e_j})$ an straight forward calculation yields if we use $a^{\sharp}(.,\xi) = \sum_{j \in N} J_{\e_j}a(.,\xi) \psi_{j}(\xi)$ and (\ref{39e}):
  \begin{align}\label{41e}
    |\pa{\alpha} D_x^{\beta}a^b(x,\xi)|
    =\left| \varphi_{\e_j} \ast \left\{ \sum_{j \in N} \pa{\alpha} \left[ D_x^{\beta} a(., \xi) \psi_j(\xi)\right] \right\}(x) \right|
  \end{align}
  Our task is to use the previous equality in order to show for $\tilde{\e} \in \left( 0, (\gamma-\delta)\tau \right)$:
  \begin{align}\label{40e}
    |\pa{\alpha} D_x^{\beta}a^b(x,\xi)| \leq C_{\alpha, \beta} (x)\<{\xi}^{-m-\delta|\beta| + (\gamma-\delta)(\tilde{m}-|\beta|+\tau) -\tilde{\e} + \rho|\alpha|} \xrightarrow{|x| \rightarrow \infty} 0.
  \end{align}
  Then a combination of (\ref{31e}) and (\ref{32e}) and (\ref{40e}) yields claim $ii)$ and $iii)$. It remains to verify (\ref{40e}). The properties of the Fourier transform imply $g_{\e_j}, g \in \s$ for all $j \in \N_0$. Consequently $\<{y}^{n+1} g_{\e_j}(y) \in \sindo{y}$ for all $j \in N$. On account of the choice of $a$ we get using (\ref{7e}):
  \begin{align}\label{34e}
    \sum_{j \in N} \left| \pa{\alpha} \left\{ D_x^{\beta} a(x,\xi) \psi_j(\xi) \right\} \right|
    \leq A_1 \<{\xi}^{m-\rho|\alpha| + \delta |\beta|}
  \end{align}
  where $A_1$ is independent of $x,\xi \in \Rn$. Due to  $\<{y}^{n+1} g_{\e_j}(y) \in \sindo{y}$ for all $j \in N$ we can choose an $R>1$ such that for $A_2:= \intr \<{y}^{-n-1} dy$ we have
  \begin{align}\label{33e}
    \left| \<{y}^{n+1} g_{\e_j}(y)\right| < \frac{\e}{2A_1A_2} \qquad \text{for all } y \in \Rn \backslash \overline{B_{R-1}(0)} \text{ and } j \in N.
  \end{align}
  In addition we choose an $\eta \in C^{\infty}_0(\Rn)$ such that $\eta(x) \in [0,1]$, $\eta(x)=1$ if $|x| \leq R-1$ and $\eta(x)=0$ if $|x| \geq R$. Then we obtain for all $x \in \Rn$ by means of Lemma \ref{lemma:PropertiesOfJ_E} iv), (\ref{34e}) and (\ref{33e}):
  \begin{align}\label{35e}
    &\left| \left[ \varphi_{\e_j} (1-\eta)\ast \sum_{j \in N} \pa{\alpha} \left\{ D_x^{\beta} a(., \xi) \psi_j(\xi) \right\} \right](x)\right|\nonumber\\
    &\quad \qquad \leq \int_{\Rn \backslash \overline{B_{R-1}(0)} } |\varphi_{\e_j}(y)| |(1-\eta)(y)| \cdot \left\| \sum_{j \in N} \pa{\alpha} \left\{ D_x^{\beta} a(., \xi) \psi_j(\xi) \right\}(x-y) \right\|_{L^{\infty}(\Rn_x)} dy \nonumber\\
    &\quad \qquad \leq \frac{\e}{2} \<{\xi}^{m-\rho|\alpha|+\delta|\beta|}.
  \end{align}
  On account of the properties of the Fourier transform and due to the definition of $\varphi_{\e_j}$ we get using $g \in \s$:
  \begin{align}\label{37e}
    \int_{\Rn \backslash \overline{B_{R-1}(0)}} |\varphi_{\e_j}| dy \leq  \intr \e_j^{-n} \left| g\left( \frac{y}{\e_j} \right) \right| dy =\intr |g(z)| dz =: B_1 <\infty,
  \end{align}
  where $B_1$  is independent of $j \in \N$. 
  The choice of the symbol $a$ and the multi-index $\beta$ gives us the existence of an $\tilde{R}>0$ such that for all $|x| \geq \tilde{R}+R-1$ and for all $y \in \overline{B_{R-1}(0)}$ we have
  \begin{align}\label{42e}
    \left|  \sum_{j\in N} \pa{\alpha} \left\{ D_x^{\beta} a(., \xi) \psi_j(\xi) \right\}(x) \right|
    \leq \frac{\e}{2B_1} \<{\xi}^{m-\rho|\alpha|+\delta|\beta|}.
  \end{align} 
  Using (\ref{41e})  we obtain for all $x \in \Rn$ with $|x| \geq \tilde{R} + R-1$:
  \begin{align*}
    |\pa{\alpha} D_x^{\beta} a^b(x,\xi)|
     &\leq \left| \intr \e_j^{-n}g(\frac{y}{\e_j}) \eta(y) \sum_{j \in N} \pa{\alpha} \left[ D_x^{\beta} a(x-y, \xi) \psi_j(\xi) \right] dy \right| \\
	& \quad + \left| \sum_{j \in N} \pa{\alpha} \left[ D_x^{\beta} a(x-y, \xi) \psi_j(\xi) \right]  \right| \\
 	 & \quad + \left| \intr \varphi_{\e_j}(y) [1-\eta](y) \sum_{j \in N} \pa{\alpha} \left[ D_x^{\beta} a(x-y, \xi) \psi_j(\xi) \right] dy  \right|.
  \end{align*}
    Now we use (\ref{35e}) in order to estimate the second summand of the previous inequality. The integrand of the first summand is always $0$ if $|y| \geq R$. Hence we can estimate the first summand of the previous inequality by means of (\ref{42e}) and (\ref{37e}). Then we get $|\pa{\alpha} D_x^{\beta} a^b(x,\xi)| \leq \e \<{\xi}^{m-\rho|\alpha|+\delta|\beta|}$ for all $x\in \Rn$ with $|x| \geq \tilde{R}+R-1$. Hence 
    \begin{align}\label{43e}
      |\pa{\alpha} D_x^{\beta} a^b(x,\xi)|  \<{\xi}^{-m+\rho|\alpha|-\delta|\beta|} \leq C_{\alpha,\beta}(x) \xrightarrow{|x| \rightarrow \infty } 0.
    \end{align}
   Now let $\tilde{\e}$ be as in the assumptions. Setting $\theta:= \frac{(\gamma-\delta)(\tilde{m}-|\beta| +\tau)-\tilde{\e} }{(\gamma-\delta)(\tilde{m}-|\beta| +\tau)}$ we get by means of interpolation with (\ref{32e}) and (\ref{43e}), that estimate (\ref{40e}) holds:
   \begin{align*}
      |\pa{\alpha} D_x^{\beta} a^b(x,\xi)|  \<{\xi}^{-m+ (\gamma-\delta)(\tilde{m}-|\beta|+\tau)-\tilde{\e}+\rho|\alpha|-\delta|\beta|} \leq C_{\alpha, \beta}(x)^{1-\theta} C_{\alpha, \tilde{m}, \tau}^{\theta} \xrightarrow{|x| \rightarrow \infty} 0.
   \end{align*}
  Hence the lemma is proved. 
\end{proof}

\begin{lemma}\label{lemma:SymbolSmoothing2}
  Let $0 \leq \delta < \rho \leq 1$, $\tilde{m} \in \N$, $0<\tau <1$, $M \in \N \cup \{ \infty \}$, $m \in \R$ and $a \in C^{\tilde{m}, \tau} S^m_{\rho, \delta}(\RnRn; M)$. Moreover let $\gamma \in (\delta, \rho)$. Then we have for all $\beta \in \Non$ with $|\beta| \leq \tilde{m}$: 
  \begin{itemize}
      \item[i)] $D_x^{\beta} a^{\sharp}(x, \xi) \in S^{m+\delta|\beta|}_{\rho, \gamma}(\RnRn; M)$,
      \item[ii)] if $a \in C^{\tilde{m}, \tau} \dot{S}^m_{\rho, \delta}(\RnRn; M)$ or if $|\beta| \neq 0$ and $a \in C^{\tilde{m}, \tau} \tilde{S}^m_{\rho, \delta}(\RnRn; M)$ then $D_x^{\beta} a^{\sharp}(x, \xi) \in \dot{S}^{m+\delta|\beta|}_{\rho, \gamma}(\RnRn; M)$.
  \end{itemize}
\end{lemma}
\begin{proof}
  Note, that we get by means of the generalized Young inequality
  \begin{align*}
    \|\phi(\e D_x)\|_{\mathscr{L}(L^{\infty}(\Rn))} = \sup_{\|f\|_{\infty} \leq 1} \|\mathscr{F}^{-1} (\phi(\e.)) \ast f\|_{\infty} \leq C \qquad \text{for all } \e \in (0,1]. 
  \end{align*}
  Now let $\beta \in \Non$ with $|\beta| \leq \tilde{m}$. We show, that for all $\tilde{\beta}, \alpha \in \Non$ with $|\alpha| \leq M$
  \begin{align}\label{6e}
    \|D^{\tilde{\beta}}_x \pa{\alpha} D_x^{\beta} a^{\sharp} (., \xi)\|_{\infty} \leq C_{\alpha, \tilde{\beta}, \beta} \<{\xi}^{m + \delta|\beta|-\rho|\alpha| + \gamma|\tilde{\beta}| } \qquad \text{for all } \xi \in \Rn.
  \end{align}
  This implies claim i). First of all we verify (\ref{6e}) for $\tilde{\beta} \in \Non$ with $|\tilde{\beta}| \leq \tilde{m}- |\beta|$. To this end we choose an arbitrary $\xi \in \Rn$ with $N:= \{j \in \N_0: \xi \in \supp \psi_j \}$ . Then $\sharp N \leq 5$. Using $a^{\sharp}(.,\xi) = \sum_{j \in N} J_{\e_j}a(.,\xi) \psi_{j}(\xi)$, the Leibniz rule, (\ref{7e}) and Lemma \ref{lemma:InterpolationResult} yields for $\theta:=\frac{|\tilde{\beta}|}{\tilde{m}+\tau-|\beta|}$
  \begin{align}\label{8ee}
    &\|D^{\tilde{\beta}}_x \pa{\alpha} D_x^{\beta} a^{\sharp} (., \xi)\|_{\infty} 
    \leq C_{\alpha} \sum_{j \in N} \sum_{\alpha_1+\alpha_2 = \alpha} \<{\xi}^{-\rho|\alpha_2|} \|\p_{\xi}^{\alpha_1} D_x^{\beta} a(., \xi)\|_{C^{|\tilde{\beta}|}_b(\Rn)} \nonumber\\
    &\qquad \qquad \leq  C_{\alpha}  \sum_{\alpha_1+\alpha_2 = \alpha} \<{\xi}^{-\rho|\alpha_2|} \|\p_{\xi}^{\alpha_1} D_x^{\beta} a(., \xi)\|^{1-\theta}_{C^{0}_b(\Rn)} \|\p_{\xi}^{\alpha_1} D_x^{\beta} a(., \xi)\|_{C^{\tilde{m}-|\tilde{\beta}|, \tau}(\Rn)}^{\theta} \nonumber\\
    &\qquad \qquad \leq  C_{\alpha, \tilde{\beta}, \beta} \<{\xi}^{m + \delta|\beta|-\rho|\alpha| + \gamma|\tilde{\beta}| },
  \end{align}
  where $C_{\alpha, \tilde{\beta}, \beta}$ is independent of $\xi \in \Rn$.
  Now let $\tilde{\beta} \in \Non$ with $|\tilde{\beta}| + |\beta|\geq \tilde{m}$. Using $a^{\sharp}(.,\xi) = \sum_{j \in N} J_{\e_j}a(.,\xi) \psi_{j}(\xi)$, the Leibniz rule and (\ref{7e}) again, we obtain
  \begin{align*}
    \|D^{\tilde{\beta}}_x \pa{\alpha} D_x^{\beta} a^{\sharp} (., \xi)\|_{\infty} 
    \leq C_{\alpha} \sum_{j \in N} \sum_{\alpha_1+\alpha_2 = \alpha} \<{\xi}^{-\rho|\alpha_2|} \| D^{\tilde{\beta}}_x J_{\e_j} \pa{\alpha_1} D_x^{\beta} a(., \xi) \|_{\infty}.
  \end{align*}
  Now we can prove  (\ref{6e}) by means of the previous inequality since $ \pa{\alpha_1} D_x^{\beta} a(., \xi) \in C^{\tilde{m} - |\beta|, \tau} (\Rn)$ using Lemma \ref{lemma:PropertiesOfJ_E} ii) and (\ref{8e}). It remains to prove claim ii). We again assume $\beta \in \Non$ with $|\beta| \leq \tilde{m}$. Moreover let $a \in C^{\tilde{m}, \tau} \dot{S}^m_{\rho, \delta}(\RnRn; M)$ or $|\beta| \neq 0$ and $a \in C^{\tilde{m}, \tau} \tilde{S}^m_{\rho, \delta}(\RnRn; M)$. 

  Similarly to the proof of (\ref{43e}) we will now show for $\alpha, \tilde{\beta} \in \Non$ with $|\alpha| \leq M$ and $|\tilde{\beta}| \leq \tilde{m}- |\beta|$:
  \begin{align}\label{9e}
    |D^{\tilde{\beta}}_x \pa{\alpha} D_x^{\beta} a^{\sharp} (., \xi)| \leq C_{\alpha, \beta, \tilde{\beta}}(x) \<{\xi}^{m + \delta|\beta|-\rho|\alpha| + \gamma|\tilde{\beta}| } \qquad \text{for all } x,\xi \in \Rn.
  \end{align}
  Here $ C_{\alpha, \beta, \tilde{\beta}}(x)$ is bounded and $ C_{\alpha, \beta, \tilde{\beta}}(x) \xrightarrow{|x| \rightarrow  \infty} 0$.
  In order to prove (\ref{9e}) for $\alpha, \tilde{\beta} \in \Non$ with $|\alpha| \leq M$ and $|\tilde{\beta}| +  |\beta| \geq \tilde{m}$ we choose an arbitrary but fixed $\xi \in \Rn$ and  define $N$ as before. Additionally  let $\e>0$ be arbitrary. Since $a \in C^{\tilde{m}, \tau} S^m_{\rho, \delta}(\RnRn; M) $ we get by means of the Leibniz rule and by (\ref{7e}) the existence of a constant $A_1>0$ with
  \begin{align}\label{11e}
    \sum_{j \in N} |\pa{\alpha} \left\{ D_x^{\beta} a(x,\xi) \psi_j(\xi) \right\}| \leq A_1 \<{\xi}^{m-\rho|\alpha| + \delta|\beta|}.
  \end{align}
Defining $g(\xi):= \xi^{\tilde{\beta}} \phi(\xi)$ for all $\xi \in \Rn$  we obtain for all $j \in \N$ and $f \in C^0_b(\Rn)$ due to the Theorem of Fubini: 
  \begin{align}\label{10e}
    \e^{|\tilde{\beta}|} D_x^{\tilde{\beta}} J_{\e_j}(D_x)f(x)
    = \intr \mathscr{F}^{-1}_{\xi \rightarrow x} \left[ g(\e_j \xi)\right](x-y)f(y)dy.
  \end{align}
  Since $\phi(\e_j \xi) \in \mathcal{S} (\Rn_{\xi})$, there is an $R>1$ such that for all $|y| \geq R-1$
  \begin{align}\label{12e}
    |\mathscr{F}^{-1}_{\xi \rightarrow x} \left[ g(\e_j \xi)\right](y) \<{y}^{n+1}| < \frac{\e}{2A_1A_2} \qquad \text{for all } j \in N,
  \end{align}
  where $A_2:= \int \<{y}^{-n-1} dy$. Moreover we get on account of the properties of the Fourier transformation, change of variable and due to $g \in \s$: 
  \begin{align}\label{13e}
    B_3:= \intr |\mathscr{F}^{-1}_{\xi \rightarrow x} \left[ g(\e_j \xi)\right](y) | dy = \intr |\mathscr{F}^{-1}[g](z)| dz < \infty
  \end{align}
  The choice of the symbol $a$ and of the multi-index $\beta$ gives us the existence of an $\tilde{R}>0$ such that for all $|x| \geq \tilde{R}+R-1$ and for all $y \in \overline{B_{R-1}(0)}$ we have
  \begin{align}\label{14e}
    \left| \sum_{j \in N} \pa{\alpha} \{ D_x^{\beta} a(x-y,\xi) \psi(\xi) \}\right| \leq \frac{\e}{2B_3} \<{\xi}^{m-\rho |\alpha| + \delta|\beta|} \qquad \text{for all } |\beta| \neq 0
  \end{align}

  Now let $\eta \in C^{\infty}_0(\Rn)$ with $\eta(x) \in [0,1]$ for all $x \in \Rn$, $\eta(x)=0$ for all $|x| \geq R$ and $\eta(x)=1$ for all $|x| \leq R-1$.
  By means of (\ref{11e}) and (\ref{12e}) we have
  \begin{align}\label{15e}
    B_1&:= \int_{\Rn \backslash \overline{B_{R-1}(0)}} | \mathscr{F}^{-1}_{\xi \rightarrow x} [g(\e_j \xi)](y)| |(1-\eta)(y)| |\sum_{j \in N} \pa{\alpha} [D_x^{\beta} a(x-y, \xi) \psi_j(\xi)]| dy \nonumber\\
    & \leq \frac{\e}{2} \<{\xi}^{m-\rho|\alpha| + \delta |\beta|}.
  \end{align}
  Additionally a combination of (\ref{13e}) and (\ref{14e}) yields
  \begin{align}\label{16e}
    B_2 &:= \int_{B_R(0)} |\mathscr{F}^{-1}_{\xi \rightarrow x} \left[ g(\e_j \xi)\right](y) | |\eta(y)| |\sum_{j \in N} \pa{\alpha} [D_x^{\beta} a(x-y, \xi) \psi_j(\xi)]|dy \notag\\
      &\leq \frac{\e}{2} \<{\xi}^{m-\rho|\alpha| + \delta |\beta|}.
  \end{align}
  Using $a^{\sharp}(.,\xi) = \sum_{j \in N} J_{\e_j}a(.,\xi) \psi_{j}(\xi)$, (\ref{10e}) and the definition of $\e_j$ first and (\ref{15e}),  (\ref{16e}) and (\ref{8e}) afterwards, we obtain
  \begin{align*}
    |D^{\tilde{\beta}}_x \pa{\alpha} D_x^{\beta} a^{\sharp} (x, \xi)|
    &= \e_j^{-|\tilde{\beta}|} \left| \e_j^{|\tilde{\beta}|} D_x^{\tilde{\beta}} J_{\e_j} \left\{ \sum_{j \in N} \pa{\alpha} \left[ D_x^{\beta} a(x,\xi) \psi_j(\xi)\right] \right\} \right|
    \leq 2^{j\gamma|\tilde{\beta}|} (B_1+B_2)\\
    &\leq 2^{j\gamma|\tilde{\beta}|}  \e C \<{\xi}^{m + \delta |\beta| -\rho|\alpha| }
    \leq \e C \<{\xi}^{m + \delta |\beta| -\rho|\alpha| + \gamma |\tilde{\beta}|}.
  \end{align*}
  Hence (\ref{9e}) also holds for $\alpha, \tilde{\beta} \in \Non$ with $|\alpha| \leq M$ and $|\tilde{\beta}| +  |\beta| \geq \tilde{m}$. This provides ii).
\end{proof}

\begin{lemma}\label{lemma:SymbolSmoothing1}
  Let $0 \leq \delta < \rho \leq 1$, $\tilde{m} \in \N$, $0<\tau <1$, $M \in \N \cup \{ \infty \}$, $m \in \R$ and $a \in C^{\tilde{m}, \tau} \tilde{S}^m_{\rho, \delta}(\RnRn; M)$ such that
  \begin{align*}
    a(x, \xi) \xrightarrow{|x| \rightarrow \infty} a(\infty, \xi) \qquad \text{for all } \xi \in \Rn.
  \end{align*}
  Moreover we set $b(x, \xi):= a(x, \xi)-a(\infty, \xi)$ for all $x, \xi \in \Rn$. Additionally we define $a^{\sharp}, a^b, a^{\sharp}(\infty, .)$ and $a^b(\infty, .)$ as in Definition \ref{Def:SymbolSmoothing}. Then we have for $\gamma \in (\delta, \rho)$ and $\tilde{\e} \in \left( 0, (\gamma- \delta)\tau \right)$: 
  \begin{itemize}
    \item[i)] $a^{\sharp}(\infty, \xi)= a(\infty, \xi) \in S^m_{\rho, \delta}(\RnRn; 0)$,
    \item[ii)] $a^b(\infty, \xi) = 0$ for all $\xi \in \Rn$,
    \item[iii)] $a^b(x, \xi) \in C^{\tilde{m}, \tau} \tilde{S}_{\rho, \gamma}^{m-(\gamma-\delta)(\tilde{m}+\tau)+\tilde{\e}}(\RnRn; M) 
		  \cap C^{\tilde{m}, \tau} \dot{S}_{\rho, \gamma}^{m-(\gamma-\delta)(\tilde{m}+\tau)+\tilde{\e}}(\RnRn; 0) $,
    \item[iv)] $a^{\sharp}(x, \xi) = a(\infty, \xi) + b^{\sharp}(x, \xi)$ for all $x, \xi \in \Rn$.
  \end{itemize}
\end{lemma}
\begin{proof}
  First of all we verify claim $i)$. Since $a \in C^{\tilde{m}, \tau} \tilde{S}^m_{\rho, \delta}(\RnRn; M)$ we have 
  \begin{align}\label{20e}
    &\|a(x,\xi) \<{\xi}^{-m}\|_{C^{0, \tau}_b(\Rn_{\xi})} \leq \|a(x,\xi) \<{\xi}^{-m}\|_{C^1_b(\Rn_{\xi})} \leq C  \qquad \text{for all } x\in \Rn \text{ and} \notag \\
      &|a(x,\xi) \<{\xi}^{-m}| \leq C \qquad \text{for all } x,\xi \in \Rn.
  \end{align}
  Hence the definition of $C^{0, \tau}(\Rn)$ provides
  \begin{align}\label{ee1}
    |\<{\xi_1}^{-m} a(x, \xi_1)-\<{\xi_2}^{-m}a(x,\xi_2)| \leq C |\xi_1-\xi_2|^{\tau} \xrightarrow{\xi_1 \rightarrow \xi_2} 0,
  \end{align}
  where $C$ is independent of $x \in \Rn$. Taking $|x| \rightarrow \infty$ on both sides and using $\<{\xi}^{-m} \in C^{\infty}(\Rn)$ yields $ a(\infty, \xi) \in C^0(\Rn_{\xi})$.  Taking $|x| \rightarrow \infty$ on both sides of (\ref{20e}) provides 
  \begin{align*}
    |a(\infty, \xi) | \leq C \<{\xi}^{m} \qquad \text{for all } \xi \in \Rn.
  \end{align*}
Together with (\ref{ee1}) we therefore get
  \begin{align*}
    a(\infty ,\xi) \in S^m_{\rho, \delta}(\RnRn; 0).
  \end{align*}
  By means of Remark \ref{bem:PropOsciIntBasic} we can show for all $\xi \in \Rn$
  \begin{align*}
    J_{\e}a(\infty, \xi)
    = a(\infty, \xi)\cdot \osint e^{-iz \cdot \eta} \phi(\e \eta) dz \dq \eta 
    =a(\infty, \xi).
  \end{align*}
  Hence we obtain for all $\xi \in \Rn$
  \begin{align*}
    a^{\sharp}(\infty, \xi) = a(\infty, \xi) \qquad \text{ and }\qquad a^b(\infty, \xi)= a(\infty, \xi)-  a^{\sharp}(\infty, \xi) = 0.
  \end{align*}
  This provides $i)$, $ii)$ and $iv)$. It remains to verify claim $iii)$. On account of the definition of $a(\infty, \xi)$ and $a \in C^ {\tilde{m}, \tau} \tilde{S}^m_{\rho, \delta}(\RnRn; M)$ we have for all $\beta \in \Non$ with $|\beta| \leq \tilde{m}$
  \begin{align}\label{21e}
    |D_x^{\beta}b(x,\xi)| \leq C_{\beta}(x)\<{\xi}^{m+\delta|\beta|} \qquad \text{for all } \xi \in \Rn,
  \end{align}
  where $C_{\beta}(x) \rightarrow 0$ if $|x| \rightarrow \infty$. Moreover
  \begin{align*}
    \|a(\infty,\xi)\|_{C^{\tilde{m}, \tau}(\Rn_x)} \leq |a(\infty,\xi)| 
    = |\lim_{|x| \rightarrow \infty} a(x,\xi)| 
    \leq C \<{\xi}^{m} \qquad \text{for all } \xi \in \Rn,
  \end{align*}
  we get $ \|b(.,\xi)\|_{C^{\tilde{m}, \tau}(\Rn)} \leq \<{\xi}^{m + \delta \cdot(\tilde{m}+\tau)}$. Together with (\ref{21e}) this yields 
  \begin{align*}
    b(x, \xi) \in C^{\tilde{m}, \tau} \dot{S}^m_{\rho, \delta}(\RnRn; 0).
  \end{align*}
  Consequently Lemma \ref{lemma:SymbolSmoothing3} and $a^b(x, \xi)= b^b(x, \xi)$ provides claim $iii)$.
\end{proof}

\subsection{Symbol Reduction}\label{Subsection:SymbolReduction}

In this subsection we prove a formula representing an operator with a non-smooth double symbol as an operator with a non-smooth single symbol. Non-smooth double symbols are defined in the following way:

\begin{Def}
  Let $\tilde{m} \in \N_0$, $0 <\tau < 1$, $m_1, m_2 \in \R$, $0 \leq \delta, \rho \leq 1$ and $M_1, M_2 \in \N_0 \cup \{ \infty \}$. Then a continuous function $a: \R^n_x \times\R^n_{\xi} \times \R^n_{x'} \times\R^n_{\xi'} \rightarrow \C$ belongs to the non-smooth double symbol-class $C^{\tilde{m}, \tau} S^{m_1, m_2}_{\rho, \delta}(\RnRnRnRn; M_1, M_2)$ if
	\begin{itemize}
		\item[i)] $\pa{\alpha} \p^{\beta'}_{x'} \p_{\xi'}^{\alpha'} a \in C^{\tilde{m}, \tau}(\R^n_x)$ and $\p_x^{\beta} \pa{\alpha} \p^{\beta'}_{x'} \p_{\xi'}^{\alpha'} a \in C^{0}(\R^n_x \times \R^n_{\xi} \times \R^n_{x'} \times\R^n_{\xi'})$,
		\item[ii)] $\left| \p^{\beta}_{x} \pa{\alpha}  \p^{\beta'}_{x'} \p_{\xi'}^{\alpha'} a(x,\xi, x', \xi')  \right| \leq C_{\alpha,\beta, \beta', \alpha'}(x) \tilde{C}_{\alpha,\beta, \beta', \alpha'}(x') \<{\xi}^{m_1-\rho|\alpha|+ \delta|\beta|} \<{\xi'}^{m_2-\rho|\alpha'|} $ \\ $\hspace*{1mm}\qquad  \qquad \qquad \qquad \qquad \qquad \cdot\<{\xi; \xi'}^{\delta |\beta'|}$
		\item[iii)] $\| \pa{\alpha}  \p^{\beta'}_{x'} \p_{\xi'}^{\alpha'} a(.,\xi, x', \xi')  \|_{C^{\tilde{m}, \tau}(\R^n)} \leq C_{\alpha, \beta', \alpha'} \<{\xi}^{m_1-\rho|\alpha|+ \delta(\tilde{m}+\tau)} \<{\xi'}^{m_2-\rho|\alpha'|} \<{\xi; \xi'}^{\delta |\beta'|}$
	\end{itemize}
  for all $x,\xi, x', \xi' \in \Rn$ and arbitrary $\beta, \alpha, \beta', \alpha' \in \N_0^n$ with $|\beta| \leq \tilde{m}$,  $|\alpha| \leq M_1$ and $|\alpha'| \leq M_2$. Here the constants $C_{\alpha,\beta, \beta', \alpha'}(x), C_{\alpha, \beta', \alpha'}$ and $\tilde{C}_{\alpha,\beta, \beta', \alpha'}(x')$ are bounded and independent of $\xi, x', \xi' \in \Rn$ respectively $\xi, x, \xi' \in \Rn$.\\
  If we even have $C_{\alpha,\beta, \beta', \alpha'}(x) \xrightarrow{|x| \rightarrow \infty} 0$ for all $\beta, \alpha, \beta', \alpha' \in \N_0^n$ with $|\beta| \leq \tilde{m}$,  $|\alpha| \leq M_1$ and $|\alpha'| \leq M_2$, then $a$ is an element of  $C^{\tilde{m}, \tau} \dot{S}^{m_1, m_2}_{\rho, \delta}(\RnRnRnRn; M_1, M_2)$. If we have $\tilde{C}_{\alpha,\beta, \beta', \alpha'}(x')  \xrightarrow{|x'| \rightarrow \infty} 0$ $\beta, \alpha, \beta', \alpha' \in \N_0^n$ with $|\beta| \leq \tilde{m}$,  $|\alpha| \leq M_1$ and $|\alpha'| \leq M_2$ instead, then $a$ is an element of  $C^{\tilde{m}, \tau} \hat{S}^{m_1, m_2}_{\rho, \delta}(\RnRnRnRn; M_1, M_2)$.
\end{Def}

For each double symbol  $a \in C^{\tilde{m}, \tau} S^{m_1, m_2}_{\rho, \delta}(\RnRnRnRn; M_1, M_2)$ we define the associated pseudodifferential operator $P$ by
\begin{align*}
  P u(x):=\osiint e^{-i(y\cdot \xi+y'\cdot \xi')} a(x, \xi, x+y, \xi') u(x+y+y') dy dy' \dq \xi \dq \xi'
\end{align*}
for all $u \in \s$.


In the smooth case, i.e. if $M_1, M_2=\infty$, the symbol-reduction is well-known, cf. e.g. \cite[Lemma 2.4]{KumanoGo}. For non-smooth double symbols of the symbol-class $C^{\tilde{m}, \tau} S^{m,m'}_{\rho, \delta}(\RnRnRnRn; N)$ the symbol smoothing was proved in \;\cite[Theorem 3.33]{Koeppl} in the case $N=\infty$ and in \cite[Section 4.2]{Paper1} in the case $(\rho, \delta) = (0,0)$. As an ingredient for the proof of the Fredholm property of non-smooth pseudodifferential operators, we need the symbol reduction in a more general setting.

\begin{thm}\label{thm:aLInCsSm00n}
  Let $0 < s < 1$, $\tilde{m} \in \N_0$ and $m_1,m_2 \in \R$. Additionally we choose $N_1,N_2 \in \N_0 \cup \{ \infty \}$ such that there is an $l \in \N$ with $N_1 \geq l > n$. Moreover, we define $\tilde{N}:=\min\{ N_1-(n+1), N_2\}$. Furthermore, let $\mathscr{B} \subseteq C^{\tilde{m}, s} S^{m_1,m_2}_{\rho,\delta}(\RnRnRnRn; N_1, N_2)$ be bounded. If we define for each $a \in \mathscr{B}$ and $\theta \in [0,1]$ the function $a^{\theta}_L: \RnRn \rightarrow \C$ by
  \begin{align*}
    a^{\theta}_L(x,\xi) := \osint e^{-iy \cdot \eta} a(x, \theta \eta + \xi, x+y, \xi) dy \dq \eta \qquad \text{for all } x, \xi \in \Rn,
  \end{align*}
  we get with $m:= m_1+m_2$ that $a^{\theta}_L \in  C^{\tilde{m},s} S^{m}_{\rho,\delta}(\RnRn; \tilde{N}) $ for all $a \in \mathscr{B}$ and $\theta \in [0,1]$ and the existence of a constant $C_{\alpha}$, independent of $a \in \mathscr{B}$ and $\theta \in [0,1]$, such that  for all $\alpha, \beta \in \Non \text{ with } |\alpha| \leq \tilde{N}$ and $ |\beta| \leq \tilde{m}$
  \begin{align}\label{55e}
    \| \pa{\alpha} a^{\theta}_L(.,\xi) \|_{ C^{\tilde{m}, s} (\Rn) } \leq C_{\alpha} \<{\xi}^{m - \rho|\alpha| + \delta(\tilde{m}+s)} \qquad \forall \xi \in \Rn
  \end{align}
  and 
  \begin{align}\label{56e}
    | \pa{\alpha} \p^{\beta}_x a^{\theta}_L(x,\xi) | \leq C_{\alpha, \beta}(x) \<{\xi}^{m -\rho |\alpha|+ \delta|\beta|} \qquad \forall \xi \in \Rn,
  \end{align}
  where $C_{\alpha, \beta}(x)$ is bounded and  independent of $a \in \mathscr{B}$, $\xi \in \Rn$ and $\theta \in [0,1]$.
  This implies the boundedness of $\{ a^{\theta}_L: a \in \mathscr{B}, \theta \in [0,1] \} \subseteq  C^{\tilde{m},s} S^{m}_{\rho,\delta}(\RnRn; \tilde{N}) $. If $\mathscr{B}$ is even a bounded set in $C^{\tilde{m}, s} \dot{S}^{m_1,m_2}_{\rho,\delta}(\RnRnRnRn; N_1, N_2)$ or in $C^{\tilde{m}, s} \hat{S}^{m_1,m_2}_{\rho,\delta}(\RnRnRnRn; N_1, N_2)$, then $C_{\alpha, \beta}(x) \xrightarrow{|x| \rightarrow \infty} 0$. 
\end{thm}

We combine the ideas of the smooth symbol reduction in \cite[Lemma 2.4]{KumanoGo} and that one in  \cite[Section 4.2]{Paper1} in order to get the boundedness of $\{ a^{\theta}_L: a \in \mathscr{B}, \theta \in [0,1] \} \subseteq  C^{\tilde{m},s} S^{m}_{\rho,\delta}(\RnRn; \tilde{N}) $. To show $C_{\alpha, \beta}(x) \xrightarrow{|x| \rightarrow \infty} 0$ additionally  some new arguments are needed. Unfortunately one looses some regularity with respect to the second variable of the order $n+1$ in the proof. The ability to treat the even and odd space dimensions in the same way is based on the next remark:

\begin{bem}\label{bem:OddSpaceDimensions}
  Let $l \in \N$ be arbitrary. Then
  \begin{align*} 
    e^{iy\cdot \eta} 
    =& \left\{ \left( 1 + \<{\xi}^{2 \delta} |y|^2\right)^{-(l+1)} \left( 1 + \<{\xi}^{2 \delta} (-\Delta_{\eta})\right)^{l} \right.\notag \\
      & \left. + \sum_{j=1}^n \left( 1 + \<{\xi}^{2 \delta} |y|^2\right)^{-(2l+1)/2} \frac{\<{\xi}^{ \delta} y_j}{\left(1 + \<{\xi}^{2 \delta} |y|^2 \right)^{1/2}} \left( 1 + \<{\xi}^{2 \delta} (-\Delta_{\eta})\right)^{l} \<{\xi}^{ \delta} D_{\eta_j}  \right\} e^{iy\cdot \eta}
  \end{align*}
  and we have for all $l_0 \in \N$, $\gamma \in \Non$ 
  \begin{align}\label{50e}
    |\p^{\gamma}_y \left( 1 + \<{\xi}^{2 \delta} |y|^2\right)^{-l_0}| \leq C_{l_0, \gamma} \<{\xi}^{\delta |\gamma|} \left( 1 + \<{\xi}^{2 \delta} |y|^2\right)^{-l_0} \qquad \forall y, \xi \in \Rn.
  \end{align}
  We additionally have for all $\gamma \in \Non$:
  \begin{align}\label{eq5n}
    \left| \p^{\gamma}_y \frac{\<{\xi}^{ \delta} y_j}{\left(1 + \<{\xi}^{2 \delta} |y|^2 \right)^{1/2}} \right|
    \leq \<{\xi}^{ \delta |\gamma|}.
  \end{align}
\end{bem}

\begin{Def}
  Let $l \in \N$ be arbitrary. Then we define
  \begin{align*}
    B^{l}(y, \Delta_{\eta}):= \left( 1 + \<{\xi}^{2 \delta} |y|^2\right)^{-l/2}  \left( 1 + \<{\xi}^{2 \delta} (-\Delta_{\eta})\right)^{l/2}
  \end{align*}
  if $l$ is even and 
  \begin{align*}
    B^{l}(y, \Delta_{\eta})
    :=& \left( 1 + \<{\xi}^{2 \delta} |y|^2\right)^{-l/2-1/2} \left( 1 + \<{\xi}^{2 \delta} (-\Delta_{\eta})\right)^{(l-1)/2} \notag \\
      & + \sum_{j=1}^n \left( 1 + \<{\xi}^{2 \delta} |y|^2\right)^{-l/2} \frac{\<{\xi}^{ \delta} y_j}{\left(1 + \<{\xi}^{2 \delta} |y|^2 \right)^{1/2}} \left( 1 + \<{\xi}^{2 \delta} (-\Delta_{\eta})\right)^{(l-1)/2} \<{\xi}^{ \delta} D_{\eta_j} 
  \end{align*}
  else for all $y, \xi \in \Rn$.
\end{Def}

In order to improve the symbol reduction, we  need the next result:

\begin{prop}\label{prop:OsziInt=Int_n}
  Let $0 \leq \delta \leq \rho \leq 1$ with $\delta \neq 1$, $0< \tau < 1$, $\tilde{m} \in \N_0$ and $m_1,m_2 \in \R$. Additionally let $N_1,N_2 \in \N_0 \cup \{ \infty \}$ be such that there is an $l \in \N$ with $n < l \leq N_1$. Moreover, let $a \in C^{\tilde{m},\tau}S^{m_1,m_2}_{\rho,\delta}(\RnRnRnRn; N_1,N_2)$. Considering an $l_0 \in \N_0$ with $n< l_0 \leq N_1$, we define $r^{\theta}: \RnRnRnRn \rightarrow \C$ for all $\theta \in [0,1]$ by
  \begin{align*}
    r^{\theta}(x,\xi,y,\eta):= B^{l_0}(y, \Delta_{\eta}) a(x,\xi+\theta \eta,x+y,\xi)
  \end{align*}
  for all  $x, \xi, \eta, y \in \Rn$. Then we have $r^{\theta}(x,\xi,y, \eta) \in L^1(\Rn_y)$ for all $x,\xi, \eta \in \Rn$ and  $ \int e^{-iy\cdot \eta } r^{\theta}(x,\xi,y, \eta) dy \in L^1(\Rn_{\eta})$ for all $x,\xi \in \Rn$. Moreover  we obtain 
  \begin{align*}
    \osint e^{-iy\cdot \eta } r^{\theta}(x,\xi,y,\eta) dy \dq \eta = \int \left[ \int e^{-iy\cdot \eta } r^{\theta}(x,\xi,y,\eta) dy \right] \dq \eta.
  \end{align*}
\end{prop}

\begin{proof} First of all we prove the claim for even $l_0$ and use $2l_0$ instead of $l_0$. 
   Let $x, \xi \in \Rn$ be arbitrary. We define $m:=m_1+m_2$. For every $\tilde{\gamma} \in \Non$ we get due to the boundedness of $\mathscr{B}_2 \subseteq C^{\tilde{m},\tau} S^{m_1,m_2}_{\rho,\delta}(\RnRnRnRn; N_1,N_2)$, the Leibniz rule and  $\<{\xi + \theta \eta; \xi} \leq \<{\xi} \<{\eta}$ for $\tilde{l} \in \N_0$, $\tilde{l} \leq l_0$:
  \begin{align}\label{eq7n}
    &\left| \p_y^{\tilde{\gamma} } \left\{ [\<{\xi}^{2 \delta}(-\Delta_{\eta})]^{\tilde{l}}) a(x,\xi+\theta \eta,x+y,\xi) \right\} \right| 
    \leq C_{ \tilde{l}, \tilde{\gamma} }  \<{\eta}^{|m_1| + \delta |\tilde{\gamma}|} \<{\xi}^{m + \delta |\tilde{\gamma}|+2\tilde{l}\delta }
  \end{align}
  for all $y,\eta \in \Rn$, where $C_{\tilde{l},\tilde{\gamma} }$ is independent of $x,y,\xi,\eta \in \Rn$, $\theta \in [0,1]$ and $a \in \mathscr{B}_2$. Now the Leibniz rule provides for all $l \in \N_0$ by means of (\ref{eq7n}) and (\ref{50e}) the existence of a $C_l>0$, independent of $x,y,\xi,\eta \in \Rn$, $\theta \in [0,1]$ and $a \in \mathscr{B}$, such that
  \begin{align*}
    |\<{\eta}^{-2l}\<{D_y}^{2l} r^{\theta}(x,\xi,y,\eta)| &\leq C_{l} \<{\eta}^{-2l} \left( 1 + \<{\xi}^{2 \delta} |y|^2\right)^{-l_0} \<{\eta}^{|m_1| + 2l \delta} \<{\xi}^{m + 2l\delta+ 2{l}_0\delta }\\
    &\leq C_{l} \<{y}^{-2l_0} \<{\eta}^{|m_1|-2l(1-\delta)} \<{\xi}^{m + 2l\delta +2l_0 \delta}
  \end{align*}
  for all $\xi,\eta \in \Rn$. \\ 
  Assuming an arbitrary $\chi \in \s$ with $\chi(0)=1$, we get for fixed $x,\eta, \xi \in \Rn$:
  \begin{align}\label{p61}
    e^{-iy\cdot \eta } \chi(\e y) r^{\theta}(x,\xi,y,\eta) \xrightarrow[]{\e \rightarrow 0} e^{-iy\cdot \eta } r^{\theta}(x,\xi,y,\eta) \quad \text{pointwise for all } y \in \Rn.
  \end{align}
  Now let $0 < \e \leq 1$. Using the Leibniz rule and $\chi \in \s \subseteq C^{\infty}_b(\Rn)$ we have
  \begin{align}
    | \<{\eta}^{-2l'} \<{D_y}^{2l'} [\chi(\e y) r^{\theta}(x,\xi,y,\eta)]| &\leq  C_{l} \<{y}^{-2l_0} \<{\eta}^{|m_1|-2l'(1-\delta)} \<{\xi}^{m + 2l'\delta +2l_0 \delta }, \label{p63}
  \end{align}
  for all  $l' \in \N_0$ uniformly in $x,\xi, \eta, y \in \Rn$,  $a \in \mathscr{B}_2$ and in $0 < \e \leq 1$. 
  Integration by parts yields for arbitrary $\ell \in \N_0$ with $|m_1|-2\ell(1-\delta) < -n$:
  \begin{align}\label{p64}
    \int e^{-iy\cdot \eta } \chi(\e y) r^{\theta}(x,\xi,y,\eta) dy  = \int e^{-iy\cdot \eta } \<{\eta}^{-2\ell} \<{D_y}^{2\ell} [\chi(\e y) r^{\theta}(x,\xi,y,\eta)] dy 
  \end{align}
  Using $\chi \in \s \subseteq C^{\infty}_b(\Rn) $ and (\ref{p64}) first and (\ref{p63}) afterwards provides for fixed $x,\xi \in \Rn$:
  \begin{align}\label{p65}
    \left| \chi(\e \eta) \hspace{-0.5mm} \int \hspace{-0.5mm} e^{-iy\cdot \eta } \chi(\e y) r^{\theta}(x,\xi,y,\eta) dy \right| 
    \leq C_{l,m,\xi}  \<{\eta}^{|m_1|-2\ell(1-\delta)} \in L^1(\Rn_{\eta}).
  \end{align}
  Here the constant $C_{\ell,m,\xi}$ is independent of $\e \in (0, 1]$,  $a \in \mathscr{B}_2$ and $x \in \Rn$. Setting $l'=0$ in (\ref{p63}) we obtain for each fixed $x,\xi,\eta \in \Rn$, that 
  $$\{ y \mapsto \chi(\e y) r^{\theta}(x,\xi,y,\eta): 0 < \e \leq 1 \}$$ 
  has a $L^1(\Rn_y)$-majorant. Together with (\ref{p61}) and (\ref{p65}) we have verified all assumptions of Lebesgue's theorem. An application of Lebesgue's theorem two times provides
  \begin{align*}
    \osint e^{-iy\cdot \eta } r^{\theta}(x,\xi,y,\eta) dy \dq \eta 
    = \int \left[ \int e^{-iy\cdot \eta } r^{\theta}(x,\xi,y,\eta) dy  \right]\dq \eta.
  \end{align*}
  If $l_0$ is odd, we can prove the claim in the same way, using 
  Remark \ref{bem:OddSpaceDimensions}.\\
\end{proof}

\begin{prop}\label{prop:HilfslemmaIntAbschatzung_n}
  Let $0< \delta < 1$, $m_1,m_2 \in \R$, $u \geq 0$ and $\theta \in [0,1]$. Additionally let $X$ be a Banach space. Considering $ l_0 \in \N_0$ with $-l_0 < -n$, we choose a set $\mathscr{B}$ of functions $r: \RnRnRnRn \rightarrow \C$ such that the next inequality holds for all $l \in \N_0$:
  \begin{align*}
    \| (-\Delta_y)^{l} r^{\theta}(.,\xi,y,\eta) \|_X \leq &C_{l}(x)\tilde{C}_{l}(x+y) \left( 1 + \<{\xi}^{2 \delta} |y|^2\right)^{-l_0/2} \<{\xi + \theta \eta}^{m_1} \<{\xi}^{m_2} \\
    & \cdot \<{\xi + \theta \eta; \xi}^{ 2l\delta+u}.
  \end{align*}
  Here the constants $C_l(x), \tilde{C}_{l}(x+y)$ are bounded and independent of $\xi,\eta \in \Rn$, ${\theta} \in [0,1]$ and of $r \in \mathscr{B}$. If we denote the sets $\Omega_1:= \{ \eta \in \Rn : |\eta| \leq \frac{1}{2}\<{\xi}^{\delta} \}$, $\Omega_2:= \{ \eta \in \Rn : \frac{1}{2}\<{\xi}^{\delta} \leq |\eta| \leq \frac{1}{2} \<{\xi} \}$ and $\Omega_3:= \{ \eta \in \Rn : |\eta| \geq \frac{1}{2} \<{\xi}\}$ first and define 
  \begin{align*}
    I^{\theta}_i(x) := \int \limits_{\Omega_i} \intr e^{-iy\cdot \eta } r^{\theta}(x,\xi,y,\eta) dy \dq \eta \qquad \text{for } i \in \{ 1,2,3 \}
  \end{align*}
  for arbitrary $x, \xi \in \Rn$ afterwards, then there is a constant $C(x)$, bounded and independent of $\xi \in \Rn$, ${\theta} \in [0,1]$ and $r \in \mathscr{B}$, such that
  \begin{align}\label{eq13n}
    \left\| I^{\theta}_i \right\|_{X} \leq C(x) \<{\xi}^m \qquad \text{for } i \in \{ 1,2,3 \}.
  \end{align}
  where $m:=m_1+m_2+u$. If $X= \R$ and $C_l(x) \xrightarrow{|x| \rightarrow \infty} 0$ or $\tilde{C}_l(x+y) \xrightarrow{|x+y| \rightarrow \infty} 0$ for all $l \in \N$, then $C(x) \xrightarrow{|x| \rightarrow \infty} 0$. 
\end{prop}

\begin{proof}
  First of all we prove the claim for even $l_0$ and use $2l_0$ instead of $l_0$. 
  Let $\xi \in \Rn$. The assumptions and $ \<{\xi+ \theta \eta ; \xi} \leq \<{\xi} \<{\eta}$ give us the existence of bounded constants $C_l(x)$, $\tilde{C}_l(x+y)$, independent of $\xi,\eta \in \Rn$ and $r \in \mathscr{B}$, such that
  \begin{align}\label{p25n}
   &\| (-\Delta_y)^{l} r^{\theta}(.,\xi,y,\eta) \|_X \nonumber \\
    &\qquad \qquad \leq C_{l}(x) \tilde{C}_l(x+y)\left( 1 + \<{\xi}^{2 \delta} |y|^2\right)^{-l_0} \<{\xi + \theta \eta}^{m_1} \<{\xi}^{m_2 }\<{\xi + \theta \eta; \xi}^{ 2l\delta + u} \nonumber\\
     & \qquad \qquad \leq C_{l}(x) \tilde{C}_l(x+y) \<{y}^{-2l_0} \<{\xi}^{m + 2l \delta + u} \<{\eta}^{|m_1|+ 2l \delta + u} \in L^1(\Rn_y) 
  \end{align}
  for all $\xi,\eta \in \Rn, \theta \in [0,1], l \in \N_0$. 
    For all $\eta \in \Omega_1 \cup \Omega_2$ and $m \in \R$ the estimates
  $\<{\xi +\theta \eta}^{m_1} \leq C_{m_1} \<{\xi}^{m_1}$ and $\<{\xi + \theta \eta; \xi}^{ 2l\delta +u} \leq C \<{\xi}^{2l \delta+u}$ hold. Now let $m :=m_1+m_2+u$. Then we can simplify (\ref{p25n}) for all $\eta \in \Omega_1 \cup \Omega_2$ to 
  \begin{align}\label{p30n}
    \|(-\Delta_y)^{l} r^{\theta}(.,\xi,y,\eta) \|_{X} \leq  C_{l}(x)\tilde{C}_l(x+y) \left( 1 + \<{\xi}^{2 \delta} |y|^2\right)^{-l_0}  \<{\xi}^{m + 2l\delta}
  \end{align}
  for all $\xi,y \in \Rn,  l \in \N_0$, where $C_l(x), \tilde{C}_l(x+y)$ are bounded and independent of $\theta \in [0,1]$, $\xi,\eta \in \Rn$ and $r \in \mathscr{B}$. In order to estimate $\|I_1\|_{X}$, we also need the following calculation, which can be verified by means of the change of variables $\tilde{\eta}:= \<{\xi}^{-\delta}\eta$:
  \begin{align}\label{p26n}
    \int \limits_{|\eta| \leq 0.5 \<{\xi}^{\delta} } \dq \eta= \<{\xi}^{\delta n} \int \limits_{|\tilde{\eta}| \leq 0.5  } \dq \eta \leq C_n \<{\xi}^{\delta n } .
  \end{align}
  Thus a combination of (\ref{p30n}) and (\ref{p26n}) 
  concludes together with a substition $w:= \<{\xi}^{\delta}y$:
  \begin{align*}
    \| I_1^{\theta} \|_{X} 
    &\leq C_1(x) \<{\xi}^{m-\delta n} \int \limits_{\Omega_1} \intr \tilde{C}_0(x+\<{\xi}^{-\delta}w)\left( 1 +  |w|^2\right)^{-l_0} dw \dq \eta 
    \leq C_1(x) \<{\xi}^m,
  \end{align*}
  where $C_1(x)$ is bounded and independent of $\xi \in \Rn$ and $r \in \mathscr{B}$. For the estimate of $\|I_2\|_{X}$ and $\|I_3\|_{X}$ we choose $l \in \N_0$ with $-2l < -n$. Together with the equation
  $e^{-iy \cdot \eta} = |\eta|^{-2l}(-\Delta_y)^{l} e^{-iy \cdot \eta}$ we obtain by integration by parts:
  \begin{align}\label{p27n}
    \intr e^{-iy \cdot \eta} r^{\theta}(x,\xi,y,\eta) dy = |\eta|^{-2l} \intr e^{-iy \cdot \eta} (-\Delta_y)^{l} r^{\theta}(x,\xi,y,\eta) dy.
  \end{align}
  Additonally we have 
  \begin{align}\label{p28n}
    \int \limits_{|\eta| \geq 0.5\<{\xi}^{\delta}} |\eta|^{-2l} \dq \eta= C_n \left|\int\limits_{0.5 \<{\xi}^{\delta}}^{\infty} r^{n-1-2l}d r \right| = C_{n,l} \<{\xi}^{(-2l+n)\delta}.
  \end{align}
  If we utilize (\ref{p27n}) and (\ref{p30n}) first, and (\ref{p28n}) afterwards, we obtain
  \begin{align*}
    \|I^{\theta}_2\|_{X} 
    &\leq C_{2}(x) \<{\xi}^{m+2l \delta -\delta n } \int \limits_{\Omega_2} |\eta|^{-2l} \intr \tilde{C}_l(x+\<{\xi}^{-\delta}w) \left( 1 + |w|^2\right)^{-l_0}  dw \dq \eta 
    \leq C_{2,l}(x) \<{\xi}^m,
  \end{align*}
  where $C_{2,l}(x)$ is bounded and independent of $\xi \in \Rn$, $\theta \in [0,1]$ and $r \in \mathscr{B}$. It remains to estimate $\|I^{\theta}_3\|_{X}$. 
  For each $\eta \in \Omega_3$, we have 
  $\<{\xi + \theta \eta} \leq \<{\xi} + |\theta \eta| \leq 3 |\eta|$ and $\<{\xi + \theta \eta; \xi} \leq \sqrt{13} |\eta|$.
  Denoting $k_+:= \max \{0,k\}$ and $k_{-}:=\min\{0,k\}$ this provides together with (\ref{p25n})  the existence of some constants $C_l(x), \tilde{C}_l(x+y)$, bounded and independent of $\xi \in \Rn$, $\eta \in \Omega_3$, $\theta \in [0,1]$ and $r \in \mathscr{B}$, such that
  \begin{align}\label{p31n}
    |\eta|^{-2l}\|(-\Delta_y)^l r^{\theta}(.,\xi,y,\eta)\|_{X} \leq &C_{l}(x) \tilde{C}_l(x+y) \left( 1 + \<{\xi}^{2 \delta} |y|^2\right)^{-l_0} |\eta|^{(m_1)_+ +u - 2l(1-\delta) } \notag\\
    &\quad \cdot \<{\xi}^{m_2} 
  \end{align}
  for all $\xi,y \in \Rn$ and $\eta \in \Omega_3$.
  Analog to the calculation of (\ref{p28n}) we get
  \begin{align}\label{p29}
    \int \limits_{\Omega_3} |\eta|^{(m_1)_+ +u - 2l(1-\delta)} \dq \eta \leq C \<{\xi}^{(m_1)_+ - 2l(1-\delta) + n -m_1} \<{\xi}^{m_1+u+ \delta n} \leq C \<{\xi}^{m_1+u+\delta n},
  \end{align}
  if we choose an $l \in \N_0$ with $-(m_1)_{-} +u - 2l(1-\delta) \leq -n$.
  Finally a combination of (\ref{p27n}), (\ref{p31n}) and (\ref{p29}) concludes similarly to the calculation of $\|I^{\theta}_2\|_{X} $:
  \begin{align*}
    \|I^{\theta}_3\|_{X} 
    \leq  C_3(x) \<{\xi}^m.
  \end{align*}
  Here $C_3(x)$ is bounded and independent of $\xi \in \Rn$, $\theta \in [0,1]$ and $r \in \mathscr{B}$. 
If $X= \R$ and $C_l(x) \xrightarrow{|x| \rightarrow \infty} 0$ for all $l \in \N$, we get by verifying the proof, that $C(x) \xrightarrow{|x| \rightarrow \infty} 0$. \\
  Now assume, that $X=\R$ and that for all $l \in \N_0$ we have $\tilde{C}_l(x+y) \xrightarrow{|x+y| \rightarrow  \infty} 0$ and $\tilde{C}_l(x+y) \leq B_l$ for all $x,y \in \Rn$. In order to verify that $C(x) \xrightarrow{|x| \rightarrow \infty} 0$ in estimate (\ref{eq13n}), we choose an arbitrary $l \in \N_0$ and $\varepsilon >0$. Additionally let  $\tilde{\varepsilon}>0$ with $-l_0<-l_0 + \tilde{\varepsilon} <-n$ be arbitrary but fixed. Defining
    $A:= \intr \<{w}^{-l_0 + \tilde{\varepsilon}} dw$
  we obtain due to  $\<{w}^{-\tilde{\varepsilon}} \in \sindo{w}$ the existence of a $R>0$ such that 
  \begin{align}\label{eq14n}
    \<{w}^{-\tilde{\varepsilon}} \leq \frac{\varepsilon}{2A B_l} \qquad \text{ for all } w \in \Rn \backslash B_R(0).
  \end{align}
  Since $\tilde{C}_l(x+y)\xrightarrow{|x+y| \rightarrow \infty}  0$, there is a $\tilde{R}>0$ such that
  \begin{align}\label{eq15n}
      \tilde{C}_l(x+y) \leq \frac{\varepsilon}{2A} \qquad \text{for all } x,y \in \Rn \text{ with } |x+y|\geq \tilde{R}.
  \end{align}
  Using (\ref{eq14n}) and (\ref{eq15n}) we obtain for all $x \in \Rn$ with $|x|\geq \tilde{R}+R$:
  \begin{align*}
    \intr \tilde{C_l}(x+ \<{\xi}^{-\delta}w) \<{w}^{-l_0} dw
    &= \int_{\Rn \backslash B_R(0)} \tilde{C_l}(x+ \<{\xi}^{-\delta}w) \<{w}^{- \tilde{\varepsilon}} \<{w}^{-l_0 + \tilde{\varepsilon}} dw \\
    &\qquad  + \int_{B_R(0)} \tilde{C_l}(x+ \<{\xi}^{-\delta}w) \<{w}^{- \tilde{\varepsilon}} \<{w}^{-l_0+ \tilde{\varepsilon}} dw 
    \leq \varepsilon.
  \end{align*}
  Using the previous estimate while verifying the norm-estimates of $\|I^{\theta}_i\|_{\R}$ for all $i \in \{1,2,3\}$ we obtain  $C(x) \xrightarrow{|x| \rightarrow \infty} 0$ in inequality (\ref{eq13n}).
  \\
  If $l_0$ is odd, we can proof the claim in the same way, using 
  Remark \ref{bem:OddSpaceDimensions}.
\end{proof}

The previous results enable us to show Theorem \ref{thm:aLInCsSm00n}, now:

\begin{proof}[Proof of Theorem \ref{thm:aLInCsSm00n}]
  We prove the claim in several steps: First we verify (\ref{56e}) in the case $|\beta|=0$. Then we show (\ref{55e}) in the case $|\beta|=0$ and $\pa{\alpha}D_x^{\beta} a_L^{\theta} \in C^0(\RnRn)$. Afterwards on can use the cases $|\beta|=0$ in order to verify (\ref{56e}) and (\ref{55e}) in the general case, which concludes the theorem. We obtain all those results by means of Proposition \ref{prop:OsziInt=Int_n} and Proposition \ref{prop:HilfslemmaIntAbschatzung_n}, which are modifications of the proofs of Proposition 4.8 and Proposition 4.6 in \cite{Paper1}. To this end we need to modify the analogous results of \cite[Section 4.2]{Paper1} as already done in the proofs of  Proposition \ref{prop:OsziInt=Int_n} and Proposition \ref{prop:HilfslemmaIntAbschatzung_n}. Note, that the generalized properties of the oscillatory integrals of Subsection \ref{subsection:SpaceOfAmplitudes} are needed for the proofs. The details are left to the reader.  
\end{proof}



\section{Fredholm Property of Non-Smooth Pseudodifferential Operators}\label{FredholmProperty}

%

The present section serves to show the main goal of this paper: The Fredholm property of non-smooth pseudodifferential operators fulfilling certain properties. For the proof of that statement we use  the following properties of non-smooth pseudodifferential operators verified by Marschall:

\begin{lemma}\label{lemma:MarschallCompactness_H2}
    Let $m\in \R$, $0\leq \delta \leq \rho \leq 1$, $M \in \N_0 \cup\{\infty\}$ with $M>\frac{n}{2}$. Moreover let $\tilde{m} \in \N_0$ and $0<\tau <1$ be such that $\tilde{m}+\tau > \frac{1-\rho}{1-\delta}\cdot \frac{n}{2}$ in case $\rho<1$. Additionally let $a \in C^{\tilde{m}, \tau}S^m_{\rho,\delta}(\RnRn; M)$ be such that 
    \begin{align*}
      \lim_{|x| + |\xi| \rightarrow \infty} (1+|\xi|)^{-m} a(x,\xi)=0.
    \end{align*}
    Then for $(1-\rho)\frac{n}{2} - (1-\delta)(\tilde{m}+\tau)<s < \tilde{m}+ \tau$
    \begin{align*}
      a(x, D_x): H^{s+m}_2(\Rn) \rightarrow H^s_2(\Rn) \qquad \text{is compact.}
    \end{align*}
\end{lemma}
\begin{lemma}\label{lemma:MarschallCompactness_Hp}
  Let $m\in \R$, $0\leq \delta \leq 1$, $1 \leq p < \infty$, $\tilde{m} \in \N_0$ and $0<\tau <1$. Moreover let $M \in \N \cup\{\infty\}$ with $M>n\cdot \max\left\{ \frac{1}{2}, \frac{1}{p} \right\}$. Additionally let $a \in C^{\tilde{m}, \tau}S^m_{1,\delta}(\RnRn; M)$ be such that 
    \begin{align*}
      \lim_{|x| + |\xi| \rightarrow \infty} (1+|\xi|)^{-m} a(x,\xi)=0.
    \end{align*}
    Then for $- (1-\delta)(\tilde{m}+\tau)<s < \tilde{m}+ \tau$
    \begin{align*}
      a(x, D_x): H^{s+m}_p(\Rn) \rightarrow H^s_p(\Rn) \qquad \text{is compact.}
     \end{align*}
\end{lemma}

Lemma \ref{lemma:MarschallCompactness_H2} and Lemma \ref{lemma:MarschallCompactness_Hp} are special cases of Theorem 3 and Theorem 4 of \cite{MarschallOnTheBoundednessAndCompactness}. 
By means of those two lemmas we obtain the next two corollaries: 

\begin{kor}\label{kor:CompactPDO_H2}
  Let $0 \leq \delta \leq \rho \leq 1$, $m \in \R$, $M>\frac{n}{2}$ and $\varepsilon>0$. Moreover let $\tilde{m} \in \N_0$ and $0 < \tau <1$ be such that $\tilde{m}+\tau> \frac{1-\rho}{1-\delta} \cdot \frac{n}{2}$ if $\rho<1$. Additionally let  $a \in C^{\tilde{m}, \tau} S^{m-\varepsilon}_{\rho, \delta}(\RnRn;M) \cap C^{\tilde{m}, \tau} \dot{S}^{m-\varepsilon}_{\rho, \delta}(\RnRn;0)$.
  Then for all $s \in \R$ with
  \begin{align*}
    (1-\rho)\cdot \frac{n}{2}-(1-\delta)(\tilde{m}+\tau)<s< \tilde{m}+\tau,
  \end{align*}
  the operator
  \begin{align*}
    a(x,D_x):H^{m+s}_2(\Rn) \rightarrow H^s_2(\Rn) \qquad \text{is compact.}
  \end{align*}
\end{kor}
\begin{proof}
  Since $a \in  C^{\tilde{m}, \tau}\dot{S}^{m-\varepsilon}_{\rho, \delta}(\RnRn;0)$ implies $|a(x,\xi)|\<{\xi}^{-m}\xrightarrow{|x|+|\xi|\rightarrow \infty} 0$, the claim is a consequence of Lemma \ref{lemma:MarschallCompactness_H2}.
\end{proof}

\begin{kor}\label{kor:CompactPDO_Hp}
  Let $0 \leq \delta \leq 1$, $m \in \R$, $M>n \cdot \max\{ \frac{1}{2}, \frac{1}{p}\}$ where $1<p<\infty$ and $\varepsilon>0$. Moreover let $\tilde{m} \in \N_0$ and $0 < \tau <1$. Additionally let $a \in C^{\tilde{m}, \tau} S^{m-\varepsilon}_{1, \delta}(\RnRn;M) \cap C^{\tilde{m}, \tau} \dot{S}^{m-\varepsilon}_{1, \delta}(\RnRn;0)$.
  Then for all $s \in \R$ with
  \begin{align*}
    -(1-\delta)(\tilde{m}+\tau)<s< \tilde{m}+\tau,
  \end{align*}
  the operator
  \begin{align*}
    a(x,D_x):H^{m+s}_p(\Rn) \rightarrow H^s_p(\Rn) \qquad \text{is compact.}
  \end{align*}
\end{kor}
\begin{proof}
  Since $a \in C^{\tilde{m}, \tau} \dot{S}^{m-\varepsilon}_{\rho, \delta}(\RnRn;0)$ implies $|a(x,\xi)|\<{\xi}^{-m}\xrightarrow{|x|+|\xi|\rightarrow \infty} 0$, the claim is a consequence of Lemma \ref{lemma:MarschallCompactness_Hp}.
\end{proof}

In order to verify an asymptotic expansion of the product of two double symbols, we need the next theorem. It can be proved by means of the usual verifications of the similar result in the smooth case, see e.g. \cite[Theorem 3.1]{KumanoGo}. For the convenience of the reader, we give a short sketch of the proof. 

\begin{thm}\label{thm:AsymototicExpansion}
  Let $0 \leq \delta \leq \rho \leq 1$, $m_1,m_2 \in \R$, $M_1,M_2 \in \N_0 \cup \{\infty\}$ with $M_1>n+1$, $\tilde{m} \in \N_0$ and $0 < \tau <1$. For $a \in C^{\tilde{m}, \tau} S^{m_1,m_2}_{\rho, \delta}(\RnRnRnRn;M_1,M_2)$
  we define 
  \begin{align*}
    a_L(x,\xi):= \osint e^{-iy\cdot \eta} a(x, \xi+\eta,x+y,\xi)dy \dq \eta \quad \text{for all } x,\xi \in \Rn.
  \end{align*}
  Additionally we set for all $\theta \in [0,1]$ and $\gamma \in \Non$ with $|\gamma| \leq M_1-(n+1)$
  \begin{align*}
    r_{\gamma, \theta}(x, \xi):=\osint  e^{-iy\cdot \eta} \p_{\eta}^{\gamma} D_y^{\gamma} a(x, \xi+\theta \eta,x+y,\xi) dy \dq \eta \quad \text{for all } x, \xi \in \Rn.
  \end{align*}
  Moreover we define $\tilde{M}_k:= \min\{M_1-k-(n+1);M_2\}$ for all $k \leq  M_1-(n+1)$. Then we get for all $N \leq  M_1-(n+1)$, that 
  \begin{align}\label{58e}
    a_L(x, \xi)= \sum_{|\alpha|<N} \frac{1}{\alpha!} \p_{\eta}^{\alpha} D^{\alpha}_{y} a(x,\xi+\eta, x+y, \xi)|_{\eta=y=0} + R_N(x,\xi),
  \end{align}
  where
  \begin{align*}
    R_N(x,\xi):= N \cdot \sum_{|\gamma|=N} \int_0^1\frac{(1-\theta)^{N-1}}{\gamma!}r_{\gamma, \theta}(x,\xi) d\theta \in C^{\tilde{m}, \tau}S^{m_1+m_2-(\rho-\delta)\cdot N}_{\rho, \delta}(\RnRn; \tilde{M}_N)
  \end{align*}
  and
  \begin{align*}
    \left\{ r_{\gamma, \theta}(x,\xi) : |\theta|\leq 1 \right\} \subseteq  C^{\tilde{m}, \tau}S^{m_1+m_2-(\rho-\delta)\cdot N}_{\rho, \delta}(\RnRn; \tilde{M}_N) \quad \text{is bounded.}
  \end{align*}
   If $\p_{\xi}^{\gamma} D_y^{\gamma}a \in C^{\tilde{m},\tau} \hat{S}^{m_1-\rho,m_2+\delta}_{\rho, \delta}(\RnRnRnRn;M_1-1,M_2)$ for $|\gamma|=1$ then 
   $$R_N(x,\xi) \in  C^{\tilde{m}, \tau}\dot{S}^{m_1+m_2-(\rho-\delta)\cdot N}_{\rho, \delta}(\RnRn; \tilde{M}_N)$$ 
   for all $N \leq  M_1-(n+1)$.
\end{thm}
\begin{proof}
 An application of the Taylor expansion formula to the second variable of $a$ around $\xi$ and integration by parts provides
 \begin{align*}
    a_L(x, \xi) 
	      &= \sum_{|\gamma| <N} \frac{1}{\gamma!} \osint e^{-iy\cdot \eta}  D^{\gamma}_y \p_{\eta}^{\gamma}	a(x,\xi+\eta, x+y, \xi) dy \dq \eta \\
		&\qquad +N \sum_{|\gamma| =N}  \osint e^{-iy\cdot \eta}  \frac{\eta^{\gamma}}{\gamma!}	\int_0^1 (1-\theta)^{N-1} \p_{\eta}^{\gamma} a(x,\xi+\theta \eta, x+y, \xi)d \theta dy \dq \eta. 
 \end{align*}
  Next we need to exchange the oscillatory integral with the integral in the second term of the right side of the previous equality. Hence we choose an arbitrary $\chi \in \s$ with $\chi(0)=1$ and let $\gamma \in \Non$ with $|\gamma| = N$. Now let $l=n+1$ and $\tilde{l}=1+ \lceil \frac{m_1+n}{1-\delta}  \rceil$. Then we obtain due to the Theorem of Fubini and integration by parts using $e^{-iy\cdot \eta}= A^{\tilde{l}}(D_y, \eta) A^{l}(D_{\eta}, y) e^{-iy\cdot \eta}$, see (\ref{eqDefAg}) and (\ref{eqDefAu}) for the definition of $A^l(D_.,.)$, for each $\e>0$:
  \begin{align}\label{57e}
    &\intr \intr \int_0^1  e^{-iy\cdot \eta} \chi(\e y) \chi(\e\eta) \eta^{\gamma} (1-\theta)^{N-1} \p_{\eta}^{\gamma} a(x,\xi+\theta \eta, x+y, \xi)d \theta dy \dq \eta \notag\\
    &\qquad = \int_0^1 (1-\theta)^{N-1}  \intr \intr  e^{-iy\cdot \eta} A^{\tilde{l}}(D_y, \eta) A^{l}(D_{\eta}, y) \left\{ \chi(\e\eta) D_y^{\gamma} \right.\notag\\
    &\qquad \qquad \qquad \qquad \qquad \qquad \qquad \left. \left[ \chi(\e y)\p_{\eta}^{\gamma} a(x,\xi+\theta \eta, x+y, \xi) \right] \right\} dy \dq \eta d \theta.
  \end{align}
  Here the assumptions of the Theorem of Fubini and of integration by parts can be verified. Since $\chi \in \s$, $D^{\alpha}_y \chi(\e y) \rightarrow 0$ for $\e \rightarrow 0$ if $|\alpha| \neq 0$. Hence we get by interchanging the limit and the integration on account of (\ref{57e}) and since the integrand has an $L^1-$majorant:
  \begin{align*}
    & \osint e^{-iy\cdot \eta}  \frac{\eta^{\gamma}}{\gamma!}	\int_0^1 (1-\theta)^{N-1} \p_{\eta}^{\gamma} a(x,\xi+\theta \eta, x+y, \xi)d \theta dy \dq \eta\\
    &=\int_0^1 \frac{(1-\theta)^{N-1}}{\gamma!}  \intr \intr  e^{-iy\cdot \eta} A^{\tilde{l}}(D_y, \eta) A^{l}(D_{\eta}, y)  \left\{ D_y^{\gamma} 
    \p_{\eta}^{\gamma} a(x,\xi+\theta \eta, x+y, \xi)  \right\} dy \dq \eta d \theta \\
    &= \int_0^1   \frac{(1-\theta)^{N-1}}{\gamma!} \osint e^{-iy\cdot \eta}  D_y^{\gamma}  \p_{\eta}^{\gamma} a(x,\xi+\theta \eta, x+y, \xi) dy \dq \eta d \theta,
  \end{align*}
  where the last equality holds because of Theorem \ref{thm:propertiesOsciInt}. Hence (\ref{58e}) holds. The rest of the claim is a consequence of Theorem \ref{thm:aLInCsSm00n}.
\end{proof}

As a consequence of the previous theorem, we obtain

\begin{kor}\label{kor:Eigenschaft2}
  Let $\tilde{m}_1 \in \N$, $0 < \tau_1 <1$, $m_1, m_2 \in \R$, $0 \leq \delta < \rho \leq 1$; $M_1, M_2 \in \N_0 \cup \{ \infty \}$ with $M_1>n+1$. Additionally let $N:= M_1-(n+1)$. For $a_1 \in C^{\tilde{m}_1, \tau_1} S^{m_1}_{\rho, \delta}(\RnRn; M_1)$ and $a_2 \in S^{m_2}_{\rho, \delta}(\RnRn; M_2)$ we define
  \begin{align*}
    a(x, \xi):= \osint e^{-iy\cdot \eta} a_1(x, \xi+\eta) a_2(x+y, \xi) dy \dq \eta
  \end{align*}
  and for all $k \in \N$ with $k \leq N$, $\gamma \in \Non$ with $|\gamma| =N$ and $\theta \in[0,1]$ we set
  \begin{itemize}
    \item $a_1\sharp_k a_2 (x, \xi):= \sum\limits_{|\gamma| <k} \frac{1}{\gamma !} \pa{\gamma} a_1(x,\xi) D_x^{\gamma} a_2(x,\xi)$,
    \item $r_{\gamma, \theta}(x, \xi):= \osint  e^{-iy\cdot \eta} \p_{\eta}^{\gamma}  a_1(x, \xi+\theta \eta) D_y^{\gamma} a_2( x+y,\xi) dy \dq \eta$ 
  \end{itemize}
  for all $x, \xi \in\Rn$. Moreover we define $R_k: \RnRn \rightarrow \C$ as in Theorem \ref{thm:AsymototicExpansion}. Then 
  $$a(x, \xi) = a_1\sharp_k a_2 (x, \xi) + R_k(x, \xi) \qquad \text{for all } x, \xi \in \Rn$$
  and with $\tilde{M}_k:= \min \{ M_1-k+1; M_2\}$ and $\tilde{N}_k:= \min \{M_1-k-(n+1); M_2\}$ we obtain 
  \begin{itemize}
    \item $a_1\sharp_k a_2 (x, \xi) \in C^{\tilde{m}_1, \tau_1} S^{m_1+m_2}_{\rho, \delta}(\RnRn; \tilde{M}_k)$,
    \item $R_k (x, \xi) \in C^{\tilde{m}_1, \tau_1} S^{m_1+m_2-(\rho-\delta)k}_{\rho, \delta}(\RnRn; \tilde{N}_k)$.
  \end{itemize}
  In particular we have $a(x, \xi) \in  C^{\tilde{m}_1, \tau_1} S^{m_1+m_2}_{\rho, \delta}(\RnRn; \tilde{N}_1)$. If we even have $a_2 \in \tilde{S}^{m_2}_{\rho, \delta}(\RnRn; M_2)$, then  $R_k (x, \xi) \in C^{\tilde{m}_1, \tau_1} \dot{S}^{m_1+m_2-(\rho-\delta)k}_{\rho, \delta}(\RnRn; \tilde{N}_k)$ for all  $k \in \N$ with $k \leq N$. 
\end{kor}
\begin{proof}
  Since $a_1(x, \xi)a_2(y, \xi') \in C^{\tilde{m}_1, \tau_1} S_{\rho, \delta}^{m_1, m_2}(\RnRnRnRn; M_1, M_2)$ we just need to show $a_1\sharp_k a_2 (x, \xi) \in C^{\tilde{m}_1, \tau_1} S^{m_1+m_2}_{\rho, \delta}(\RnRn; \tilde{M}_k)$, the rest is a consequence of Theorem \ref{thm:AsymototicExpansion}. Let $k \in \N$ with $k \leq N$ be arbitrary and $\alpha, \beta, \gamma \in \Non$ with $|\gamma| <k$, $|\beta| \leq \tilde{m}_1$ and $|\alpha| \leq \tilde{M}_k$. The choice of $a_1$ and $a_2$ provides by means of the Leibniz rule
  \begin{align}\label{60e}
    |\pa{\alpha}D_x^{\beta} \left\{ \pa{\gamma} a_1(x, \xi) D_x^{\gamma} a_2(x, \xi) \right\}| 
    \leq C_{\alpha,\beta, \gamma}(x) \<{\xi}^{m_1+m_2-(\rho-\delta)|\gamma|-\rho|\alpha|+\delta|\beta|}   
  \end{align}
  for all  $x, \xi \in \Rn$, where $C_{\alpha,\beta, \gamma}(x) $ is bounded.
  On account of (\ref{glattesSymbolIstNichtglatt}) we know, that $D_x^{\gamma}a_2(x, \xi) \in C^{\tilde{m}_1, \tau_1}S^{m_2+\delta|\gamma|}_{\rho, \delta}(\RnRn; M_2)$. Hence an application of Lemma \ref{lemma:PropertyHoelderSpaces}, Lemma \ref{bem:AbschatzungNichtglattesSymbol} and the Leibniz rule provides
  \begin{align}\label{61e}
    \| \pa{\alpha} \left\{ \pa{\gamma} a_1(x, \xi) D_x^{\gamma}a_2(x, \xi) \right\}\|_{C^{\tilde{m}_1, \tau_1}(\Rn_x)} \leq C_{\alpha, \tilde{m}_1, \gamma} \<{\xi}^{m_1+m_2-(\rho-\delta)|\gamma|-\rho|\alpha|+\delta (\tilde{m}_1+ \tau_1)}.
  \end{align}
  A combination of (\ref{60e}) and (\ref{61e}) yields 
  \begin{align*}
    \pa{\gamma} a_1(x, \xi) D_x^{\gamma}a_2(x, \xi) &\in C^{\tilde{m}_1, \tau_1} S^{m_1+m_2-(\rho-\delta)|\gamma|}(\RnRn; \tilde{M}_k) \\
    &\subseteq  C^{\tilde{m}_1, \tau_1} S^{m_1+m_2}(\RnRn; \tilde{M}_k).
  \end{align*}
  Hence $a_1\sharp_k a_2 (x, \xi) \in C^{\tilde{m}_1, \tau_1} S^{m_1+m_2}_{\rho, \delta}(\RnRn; \tilde{M}_k)$.
\end{proof}

With the previous corollary at hand, we now can show the next statement: 

\begin{thm}\label{thm:Eigenschaften1}
  Let $\tilde{m}_1, \tilde{m}_2 \in \N_0$, $0<\tau_1, \tau_2 <1$, $m_1,m_2 \in \R$ and $0\leq \delta <\rho \leq 1$. Furthermore let $p=2$ if $\rho \neq 1$ and $1<p<\infty$ else. We choose a $\theta \notin \N_0$ with $\theta \in  \left(0,(\tilde{m}_2 + \tau_2)(\rho-\delta) \right)$, $\tilde{\e} \in \left( 0, \min\{ (\rho-\delta)\tau_2; (\rho-\delta)(\tilde{m}_2+\tau_2)-\theta; \theta)\} \right)$ and define $(\tilde{m},\tau):= (\lfloor s \rfloor, s-\lfloor s \rfloor)$, where $s:=\min\{ \tilde{m}_1 + \tau_1; \tilde{m}_2+ \tau_2- \lfloor \theta \rfloor\}$. Additionally let $M_1, M_2 \in \N_0 \cup\{\infty\}$ with $M_1>(n+1)+\lceil \theta \rceil + n\max\{ \frac{1}{2}, \frac{1}{p} \}$ and $M_2> n\cdot \max\{ \frac{1}{2}, \frac{1}{p} \}$.
  Moreover let $a_1 \in C^{\tilde{m}_1,\tau_1} S^{m_1}_{\rho,\delta}(\RnRn;M_1)$ and $a_2 \in C^{\tilde{m}_2,\tau_2} \tilde{S}^{m_2}_{\rho,\delta}(\RnRn;M_2)$ such that
  \begin{align*}
    a_2(x, \xi) \xrightarrow{|x|\rightarrow \infty} a_2(\infty, \xi) \qquad \text{for all } \xi \in \Rn.
  \end{align*}
  Then we get for each $s \in \R$ fulfilling $(1-\rho)\frac{n}{2}-(1-\delta)(\tilde{m}_2+\tau_2)+\theta+\tilde{\e} < s+m_1 < \tilde{m} + \tau_2$ and $(1-\rho)\frac{n}{2}-(1-\delta)(\tilde{m}+\tau)+\frac{\tilde{m}+\tau}{\tilde{m}_2+\tau_2}(\theta+\tilde{\e}) <s < \tilde{m}+\tau$, that
  \begin{align*}
    a_1(x, D_x)a_2(x, D_x)-(a_1 \sharp_{\lceil \theta \rceil} a_2)(x, D_x): H_p^{s+m_1+m_2}(\Rn) \rightarrow H^s_p(\Rn) \quad \text{is compact.}
  \end{align*}
  where $a_1 \sharp_{\lceil \theta \rceil}a_2(x,\xi)$ is defined as in Corollary \ref{kor:Eigenschaft2}.
\end{thm}

\begin{bem}
  If we weaken the condition of the second symbol in the previous theorem to $a_2 \in C^{\tilde{m}_2,\tau_2} S^{m_2}_{\rho,\delta}(\RnRn;M_2)$, then we can show in the same way as in the proof of Theorem \ref{thm:Eigenschaften1}, the compactness of 
  \begin{align*}
     a_1(x, D_x)a_2(x, D_x)-(a_1 \sharp_{\lceil \theta \rceil} a_2)(x, D_x): H_p^{s+m_1+m_2-\e}(\Rn) \rightarrow H^s_p(\Rn)
  \end{align*}
  for some $\e>0$. 
\end{bem}

\begin{proof}[Proof of Theorem \ref{thm:Eigenschaften1}]
  Let $1<p<\infty$ if $\rho=1$ and $p=2$ else. Setting $\gamma:= \delta + \frac{\theta+\tilde{\e}}{\tau_2+\tilde{m}_2}$ Corollary \ref{kor:Eigenschaft2} provides for $k \in \N$ with $k \leq M_1-(n+1)$ and $\tilde{M}_k:= \min \{ M_1-k+1; M_2 \}$ that the symbol $a_1\sharp_k a_2 $ 
  has the following properties if $a_2 \in \tilde{S}^{m_2}_{\rho, \delta}(\RnRn; M_2)$:
  \begin{itemize}
    \item[i)] $ a_1\sharp_k a_2 \in C^{\tilde{m}_1, \tau_1} S^{m_1 + m_2}_{\rho, \gamma}(\RnRn; \tilde{M}_k)$,
    \item[ii)] $\sigma(a_1(x, D_x)a_2(x, D_x)) -  a_1\sharp_k a_2 \in  C^{\tilde{m}_1, \tau_1} \dot{S}^{m_1 + m_2-(\rho-\delta)\cdot k}_{\rho, \gamma}(\RnRn; \tilde{N}_k)$,
  \end{itemize}
  where $\tilde{N}_k:= \min\{ M_1-k-(n+1);M_2 \}$ and 
  $$\sigma(a_1(x, D_x)a_2(x, D_x)):= \osint e^{-iy \cdot \eta} a_1(x, \xi+\eta) a_2(x+y, \eta)dy \dq \eta.$$ Now let $a_2 \in C^{\tilde{m}_2, \tau_2} \tilde{S}^{m_2}_{\rho, \delta}(\RnRn; M_2)$ be arbitrary. By means of Lemma \ref{lemma:SymbolSmoothing1} and Lemma \ref{lemma:SymbolSmoothing2} we get
  \begin{itemize}
    \item[iii)] $a_2^b \in C^{\tilde{m}_2, \tau_2} \tilde{S}_{\rho, \gamma}^{m_2-\theta}(\RnRn; M_2) \cap C^{\tilde{m}_2, \tau_2} \dot{S}_{\rho, \gamma}^{m_2-\theta}(\RnRn; 0) $,
    \item[iv)] $a_2^{\sharp} \in \tilde{S}^{m_2}_{\rho, \gamma}(\RnRn; M_2)$,
    \item[v)] $a_2(x,\xi)= a_2^b(x, \xi)  + a_2^{\sharp}(x,\xi)$ for all $x, \xi \in \Rn$,
  \end{itemize}
  Now let $s$ be as in the assumptions. Due to Corollary \ref{kor:CompactPDO_Hp} and Corollary \ref{kor:CompactPDO_H2} we know that
  \begin{align*}
    a_2^b(x,D_x): H^{s+m_1 +m_2}_p(\Rn) \rightarrow H^{s+m_1}_p(\Rn) \qquad \text{is compact}.
  \end{align*}
  On account of the boundedness of $a_1(x,D_x): H^{s+m_1}_p(\Rn) \rightarrow H^s_p(\Rn)$, see Theorem \ref{thm:BoundednessResultNonSmooth}, we obtain
  \begin{align}\label{5e}
    a_1(x,D_x) a_2^b(x,D_x): H^{s+m_1 +m_2}_p(\Rn) \rightarrow H^{s}_p(\Rn) \qquad \text{is compact}.
  \end{align}
  Then we obtain by means of the Leibniz rule,  Lemma \ref{lemma:PropertyHoelderSpaces} and $ a_2^b \in C^{\tilde{m}_2, \tau_2} \dot{S}_{\rho, \gamma}^{m_2-\theta}(\RnRn; 0)$ for all $\alpha \in \Non$ with $|\alpha| < \lceil \theta \rceil$:
  \begin{align}\label{1e}
    &\pa{\alpha} a_1(x, \xi) D_x^{\alpha} a_2^b(x,\xi) \in C^{\tilde{m}, \tau} S^{m_1+m_2-\theta}_{\rho, \gamma}(\RnRn; \min\{ M_1-|\alpha|; M_2\}) \notag \\
    &\qquad \qquad \qquad \qquad \qquad \qquad \cap C^{\tilde{m}, \tau}\dot{S}^{m_1+m_2-\theta}_{\rho, \gamma}(\RnRn; 0).
  \end{align}
  Due to (\ref{1e}), Lemma \ref{kor:CompactPDO_Hp} and Lemma \ref{kor:CompactPDO_H2} provides for all $\alpha \in \Non$ with $|\alpha| < \lceil \theta \rceil$:
  \begin{align}\label{2e}
    \left( \pa{\alpha} a_1 D_x^{\alpha} a_2^b \right)(x, D_x): H_p^{s+m_1+m_2}(\Rn) \rightarrow H_p^s(\Rn) \qquad \text{ is compact.}
  \end{align}
  Since $a_1 \in C^{\tilde{m_1}, \tau_1} S^{m_1}_{\rho, \gamma}(\RnRn; M_1)$ and $a_2^{\sharp} \in \tilde{S}^{m_2}_{\rho, \gamma}(\RnRn; M_2)$, we obtain together with (v) and (i), (ii) applied on $a_2^{\sharp}$ instead on $a_2$
  \begin{align}\label{4e}
    &a_1(x,D_x) a_2(x,D_x) - \left( a_1 \sharp_{\lceil \theta \rceil} a_2 \right)(x,D_x) \notag\\ 
    &\qquad =a_1(x,D_x)a_2^b(x,D_x) -\sum_{|\alpha| < \lceil \theta \rceil} \frac{1}{\alpha !} (\pa{\alpha} a_1 D_x^{\alpha} a_2^b)(x,D_x)  + R_{\lceil \theta \rceil }(x, D_x),
  \end{align}
  where
  \begin{align*}
    R_{\lceil \theta \rceil }(x, \xi) \in C^{\tilde{m}_1, \tau_1} \dot{S}_{\rho, \gamma}^{m_1+m_2-(\rho-\delta) \lceil \theta \rceil} (\RnRn; \tilde{N}_{ \lceil \theta \rceil}).
  \end{align*}
  Because of  Lemma \ref{kor:CompactPDO_Hp} and Lemma \ref{kor:CompactPDO_H2}, we get 
  \begin{align}\label{3e}
     R_{\lceil \theta \rceil }(x, D_x): H^{s+m_1 +m_2}_p(\Rn) \rightarrow H^{s}_p(\Rn) \qquad \text{is compact}.
  \end{align}
  A combination of (\ref{4e}), (\ref{5e}), (\ref{2e})  and (\ref{3e}) yields the claim.
\end{proof}

In order to verify the main result of our paper, we use

\begin{lemma}\label{lemma:HilfslemmaFuerBeweisDerFredholmproperty}
   Let $\tilde{m}, N \in \N$, $0<\tau <1$, $0\leq \delta <\rho \leq 1$ and $M \in \N_0 \cup\{\infty\}$. Additionally let $a \in C^{\tilde{m},\tau}\tilde{S}^{0}_{\rho,\delta}(\RnRn;M;\mathcal{L}(\C^N))$ be such that property $1)$ of Theorem \ref{thm:Fredholmproperty} hold. Moreover let $\psi \in C^{\infty}_b(\Rn)$ be such that $\psi(x)=0$ if $|x|\leq 1$ and $\psi(x)=1$ if $|x|\geq 2$. Then $b:\RnRn \rightarrow \C^{N\times N}$ defined by
  \begin{align*}
    b(x,\xi):= \psi(R^{-2}(|x|^2+ |\xi|^2))a(x,\xi)^{-1} \qquad \text{for all } x,\xi \in \Rn
  \end{align*}
  is an element of  $C^{\tilde{m}, \tau} \tilde{S}^0_{\rho, \delta}(\RnRn;M;\mathcal{L}(\C^N))$.
\end{lemma}
\begin{proof}
First we assume that $N=1$.
  We remark that $b(x,\xi)$ is $0$ if $|x|^2+|\xi|^2 \leq R^2$ and $b(x,\xi)=1$, if $|x|^2+|\xi|^2 \geq 2R^2$.
  Using property $1)$ of $a$ we can verify  
  \begin{align}\label{e2}
    \|a(.,\xi)^{-1}\|_{C^{0}(\Rn)} \leq C \qquad \text{and} \qquad
    \|a(.,\xi)^{-1}\|_{C^{0,\tau}(\Rn)} \leq C 
  \end{align}
  for all $|\xi|\geq R$. Due to the product rule we can write each derivative $\pa{\alpha}D_x^{\beta}a(x,\xi)^{-1}$ ($\alpha,\beta \in \Non$ with $|\alpha| \leq M$, $|\beta| \leq \tilde{m}$) as the sum of terms of the form
  \begin{align*}
    \pa{\alpha_1}D_x^{\beta_1}a(x,\xi)\cdot \ldots \cdot \pa{\alpha_k}D_x^{\beta_k}a(x,\xi) \cdot a(x,\xi)^{-l},
  \end{align*}
  where $\alpha_1+ \ldots + \alpha_k =\alpha$ and $ \beta_1 + \ldots + \beta_k=\beta \in \Non$, $k,l \in \N$. By means of Lemma \ref{lemma:PropertyHoelderSpaces}, inequality (\ref{e2}), property $1)$ and $a \in C^{\tilde{m},\tau}\tilde{S}^{0}_{\rho,\delta}(\RnRn;M)$ we get
  \begin{align*}
    &\| \pa{\alpha_1}D_x^{\beta_1}a(x,\xi)\cdot \ldots \cdot \pa{\alpha_k}D_x^{\beta_k}a(x,\xi) \cdot a(x,\xi)^{-l}\|_{C^{0,\tau}(\Rn_x)} \leq C_{\alpha, \beta} \<{\xi}^{-\rho |\alpha| + \delta(|\beta| + \tau)}\\
    &| \pa{\alpha_1}D_x^{\beta_1}a(x,\xi)\cdot \ldots \cdot \pa{\alpha_k}D_x^{\beta_k}a(x,\xi) \cdot a(x,\xi)^{-l}| \leq C_{\alpha, \beta}(x) \<{\xi}^{-\rho |\alpha| + \delta|\beta|}
  \end{align*}
  for all $x,\xi \in \Rn$ with $|\xi|\geq R$. Here $C_{\alpha, \beta}(x)$ is bounded and $C_{\alpha, \beta}(x) \xrightarrow{|x| \rightarrow \infty} 0$ if $|\beta| \neq 0$. Hence we obtain for all $\alpha, \beta \in \Non$ with $|\alpha| \leq M$ and $|\beta|\leq \tilde{m}$:
  \begin{align}\label{e3}
    \|\pa{\alpha} a(x,\xi)^{-1}\|_{C^{m,\tau}(\Rn_x)} &\leq  C_{\alpha, \tilde{m}} \<{\xi}^{-\rho |\alpha| + \delta(\tilde{m} + \tau)} \quad \forall \xi \in \Rn \text{ with } |\xi|\geq R, \\ \label{e4}
    |\pa{\alpha} D_x^{\beta} a(x,\xi)^{-1}| &\leq C_{\alpha, \beta}(x) \<{\xi}^{-\rho |\alpha| + \delta|\beta|}  \quad \forall x,\xi \in \Rn \text{ with } |x|^2+|\xi|^2 \geq R^2. 
  \end{align}
   Here $C_{\alpha, \beta}(x)$ is bounded and $C_{\alpha, \beta}(x) \xrightarrow{|x| \rightarrow \infty} 0$ if $|\beta| \neq 0$. 
   Now let $\alpha, \beta \in \Non$ with $|\alpha| \leq M$ and  $|\beta|\leq \tilde{m}$ be arbitrary. On account of the product rule and the definition of $\psi$, we obtain
   \begin{align}\label{e5}
      &|\pa{\alpha} D_x^{\beta} b(x,\xi)| = 0 &\qquad \text{for all } x,\xi \in \Rn \text{ with }|x|^2+ |\xi|^2\leq R^2 
   \end{align}
   Now let $\xi \in \Rn$ with $0 \leq |\xi|^2 \leq 2R^2$. Then we have for all $\alpha_1, \beta_1 \in \Non$, that  $\<{\xi}^{\rho|\alpha_1|-\delta|\beta_1|} \leq C_R$. Together with  (\ref{e3}) and (\ref{e4}) an application of the product rule and Lemma \ref{lemma:PropertyHoelderSpaces} provides
   \begin{align}\label{e10}
      \|\pa{\alpha} D_x^{\beta} b(x,\xi)\|_{C^{0,\tau}(\Rn_x)} \leq C_{\alpha, \beta, R} \<{\xi}^{-\rho|\alpha| + \delta(|\beta|+\tau)}
   \end{align}
   where $ C_{\alpha, \beta, R}$ is independent of $\xi \in \Rn$ with $0 \leq |\xi|^2 \leq 2R^2$. Moreover we obtain for all  $x, \xi \in \Rn$ with $R^2 \leq |x|^2+|\xi|^2 \leq 2R^2$:
   \begin{align}\label{e8}
      |\pa{\alpha} D_x^{\beta} b(x,\xi)| &\leq \sum_{\substack{ \alpha_1 + \alpha_2=\alpha \\ \beta_1+\beta_2=\beta}} C_{\alpha_1, \beta_1} \left|\pa{\alpha_1} D_x^{\beta_1}  \psi(R^{-2}(|x|^2+|\xi|^2)) \right| \left| \pa{\alpha_2} D_x^{\beta_2}a(x, \xi)^{-1}\right| \nonumber\\
      &\leq C_{\alpha, \beta, R}(x) \<{\xi}^{-\rho|\alpha| + \delta|\beta|},
   \end{align}
    where $ C_{\alpha, \beta, R}(x)$ is independent of $\xi \in \Rn$ with $R^2 \leq |\xi|^2 \leq 2R^2$ and bounded with respect to $x$. 
   Now let $\xi \in \Rn$ with $|\xi|^2 \geq 2R^2$. Then $\psi(R^{-2}(|x|^2+|\xi|^2))=1$. Hence we obtain by means of (\ref{e3}) 
   \begin{align}\label{e11}
     \|\pa{\alpha} D_x^{\beta} b(x,\xi)\|_{C^{0,\tau}(\Rn_x)} \leq C_{\alpha, \beta, R} \<{\xi}^{-\rho|\alpha| + \delta(|\beta|+\tau)},
   \end{align}
   where $ C_{\alpha, \beta, R}$ is independent of $\xi \in \Rn$ with $|\xi|^2 \geq 2R^2$. Moreover (\ref{e4}) implies for all $x, \xi \in \Rn$ with $|x|^2 + |\xi|^2 \leq 2R^2$
   \begin{align}\label{e12}
      |\pa{\alpha} D_x^{\beta} b(x,\xi)| =  \left| \pa{\alpha} D_x^{\beta}a(x, \xi)^{-1}\right|
      \leq C_{\alpha, \beta, R}(x) \<{\xi}^{-\rho|\alpha| + \delta|\beta|},
   \end{align}
   where $ C_{\alpha, \beta, R}(x)$ is bounded, independent of $\xi \in \Rn$ with $|\xi|^2 \geq 2R^2$ and $ C_{\alpha, \beta, R}(x) \xrightarrow{|x|\rightarrow \infty}0$ if $|\beta| \neq 0$.
   Now a combination of (\ref{e5}), (\ref{e10}), (\ref{e8}), (\ref{e11}) and (\ref{e12}) provides the claim: For all $\alpha, \beta \in \Non$ with $|\alpha| \leq N$, $|\beta| \leq \tilde{m}$ we have
   \begin{align*}
      \|\pa{\alpha} b(x,\xi)\|_{C^{\tilde{m}, \tau}(\Rn_x)} = \max_{|\gamma|\leq \tilde{m}} \|\pa{\alpha} D_x^{\gamma} b(x,\xi)\|_{C^{0, \tau}(\Rn_x)} \leq C_{\alpha, \tilde{m}, R} \<{\xi}^{-\rho|\alpha|+ \delta(\tilde{m}+\tau)}
   \end{align*}
    for all $\xi \in \Rn$ and 
    \begin{align*}
      |\pa{\alpha} D_x^{\gamma} b(x,\xi)|\leq C_{\alpha, \tilde{m}, R}(x) \<{\xi}^{-\rho|\alpha|+ \delta|\beta|} \qquad \text{for all } x, \xi \in \Rn,
    \end{align*}
    where $ C_{\alpha, \tilde{m}, R}(x)$ is bounded and $ C_{\alpha, \beta, R}(x) \xrightarrow{|x|\rightarrow \infty}0$ if $|\beta| \neq 0$.

Finally, let us consider the generell case $N\in \N$. Then the case $N=1$ implies that $\tilde{b}$ defined by $\tilde{b}(x,\xi):= \psi(R^{-2}(|x|^2+ |\xi|^2))\det (a(x,\xi))^{-1}$ for all $x,\xi\in\Rn$ is an element of $C^{\tilde{m}, \tau} \tilde{S}^0_{\rho, \delta}(\RnRn;M)$. Now the statement of the lemma easily follows from Cramer's rule and the fact that $C^{\tilde{m}, \tau} \tilde{S}^0_{\rho, \delta}(\RnRn;M)$ is closed with respect to pointwise multiplication.
\end{proof}

Using the main idea of the analog result in the smooth case, see \cite[Theorem 5.16]{KumanoGo}, we now are able to verify Theorem \ref{thm:Fredholmproperty}:

\begin{proof}[Proof of Theorem \ref{thm:Fredholmproperty}]
 First of all we assume, that $m=0$. In order to prove the claim let us choose $\psi \in C^{\infty}_b(\Rn)$ such that $\psi(x)=0$ if $|x|\leq 1$ and $\psi(x)=1$ if $|x|\geq 2$. Then $b:\RnRn \rightarrow \mathcal{L}(\C^N)$ defined by
  \begin{align*}
    b(x,\xi):= \psi(R^{-2}(|x|^2+|\xi|^2))a(x,\xi)^{-1} \qquad \text{for all } x,\xi \in \Rn
  \end{align*}
  is an element of $C^{\tilde{m}, \tau} \tilde{S}^0_{\rho, \delta}(\RnRn;M;\mathcal{L}(\C^N))$ on account of Lemma \ref{lemma:HilfslemmaFuerBeweisDerFredholmproperty}. Using Theorem 
  \ref{thm:Eigenschaften1} we obtain for all $s \in \R$ with $(1-\rho)\frac{n}{2} - (1-\delta)(\tilde{m} + \tau)+\theta+\tilde{\e} <s < \tilde{m} + \tau$ and $1<p<\infty$ with $p=2$ if $\rho \neq 1$:
  \begin{itemize}
    \item[i)] $a(x,D_x)b(x,D_x) = \op(ab)+R_1$,
    \item[ii)] $b(x,D_x)a(x,D_x) = \op(ab) + R_2$,
  \end{itemize}
  where
  \begin{align*}
    R_1, R_2: H^s_p(\Rn)^{N} \rightarrow H^s_p(\Rn)^{N} \qquad \text{are compact.}
  \end{align*}
  By means of the Leibniz formula and Lemma \ref{lemma:PropertyHoelderSpaces} we get
  \begin{align*}
    a(x,\xi)b(x, \xi)-I \in C^{\tilde{m}, \tau}S^0_{\rho, \delta}(\RnRn; M;\mathcal{L}(\C^N)).
  \end{align*}
  An application of Theorem \ref{lemma:MarschallCompactness_H2} in the case $\rho \neq 1$ and Theorem \ref{lemma:MarschallCompactness_Hp} else provides, that
  \begin{align}\label{ee10}
    \op(ab-I):H^s_p(\Rn)^{N} \rightarrow H^s_p(\Rn)^{N} \qquad \text{is compact}
  \end{align}
  for all $(1-\delta)\frac{n}{2}-(1-\delta)(\tilde{m} + \tau)+\theta + \tilde{\e} < s < \tilde{m} + \tau$, where $p=2$ if $\rho \neq 1$.
  Together with i) we obtain: 
  \begin{align*}
    a(x,D_x)b(x,D_x) 
    = \op(ab)-Id + Id +R_1
    = Id + \left[\op(ab-I) + R_1 \right],
  \end{align*}
  where 
  \begin{align*}
    \op(ab-I) + R_1:H^s_p(\Rn)^{N} \rightarrow H^s_p(\Rn)^{N} \qquad \text{is compact}
  \end{align*}
  for all $(1-\delta)\frac{n}{2}-(1-\delta)(\tilde{m} + \tau)+\theta + \tilde{\e} < s < \tilde{m} + \tau$, where $p=2$ if $\rho \neq 1$.
  Analogus we obtain on account of ii) and (\ref{ee10})
  \begin{align*}
    b(x,D_x)a(x,D_x)
    = \op(ab)-Id + Id +R_2
    = Id + \left[\op(ab-1) + R_2 \right],
  \end{align*}
  where 
  \begin{align*}
    \op(ab-I) + R_2:H^s_p(\Rn)^{N} \rightarrow H^s_p(\Rn)^{N} \qquad \text{is compact}
  \end{align*}
  for all $(1-\delta)\frac{n}{2}-(1-\delta)(\tilde{m} + \tau)+\theta + \tilde{\e} < s < \tilde{m} + \tau$, where $p=2$ if $\rho \neq 1$. This implies the claim for $m=0$. For general $m \in \R$, we use that $\<{D_x}^{m}: H^{m+s}_p(\Rn) \rightarrow H^s_p(\Rn)$ is a Fredholm operator for all $s \in \R$ since it is invertible. An application of the case $m=0$ to 
  \begin{align*}
    \tilde{a}(x,\xi):= a(x, \xi) \<{\xi}^{-m} \in C^{\tilde{m},\tau}\tilde{S}^{0}_{\rho,\delta}(\RnRn;M;\mathcal{L}(\C^N))
  \end{align*}
  yields that $\tilde{a}(x,D_x): H^s_p(\Rn) \rightarrow H^s_p(\Rn)$ is a Fredholm operator. Since the composition of two Fredholm operators is a Fredholm operator again, we finally obtain the statement of this theorem on account of 
  \begin{align*}
    a(x, D_x) = \tilde{a}(x,D_x)\text{diag}( \<{D_x}^{m}, \ldots, \<{D_x}^{m}): H^{m+s}_p(\Rn)^{N} \rightarrow H^s_p(\Rn)^{N}, 
  \end{align*}
  where $\text{diag}( \<{D_x}^{m}, \ldots, \<{D_x}^{m})$ is the $N \times N$ diagonal operator matrix with diagonal entries $\<{D_x}^{m}$.
\end{proof}


\end{document}